\documentclass[11pt]{amsart}
\usepackage{amsfonts}
\usepackage{amsmath}
\usepackage{amsthm}
\usepackage{mathrsfs}
\usepackage{amssymb}
\usepackage{bm,bbm}
\usepackage{caption}
\usepackage{tikz}
\usepackage[numbers]{natbib}
\usepackage{color}
\usepackage{commath}

\usepackage{bigints}

\numberwithin{equation}{section}

\allowdisplaybreaks

\theoremstyle{plain}
\newtheorem{theorem}{Theorem}[section]
\newtheorem{lemma}[theorem]{Lemma}
\newtheorem{proposition}[theorem]{Proposition}
\newtheorem{corollary}[theorem]{Corollary}
\theoremstyle{definition}

\newtheorem{remark}[theorem]{Remark}

\def\beqn{\begin{equation}}
\def\beqn*{$$}
\def\eeqn{\end{equation}}

\def\ms{\mathsf}

\def\P{\mathbb{P}}
\def\E{\mathbb{E}}
\def\ga{\gamma}

\newcommand{\reals}{{\mathbb R}}

\newcommand{\R}{\reals}

\newcommand{\vep}{\varepsilon}

\newcommand{\one}{{\mathbbm 1}}

\newcommand{\remove}[1]{}

\newcommand{\muti}{\mu^{\leftarrow}}

\newcommand{\bbT}{{\mathbb{T}}}
\newcommand{\mP}{\mathcal P}
\newcommand{\inDt}{\mathcal D_\ms{in}}
\newcommand{\inCc}{\mathcal C_\ms{in}}

\begin{document}

\bibliographystyle{abbrv}

%can adjust space between lines
\renewcommand{\baselinestretch}{1.05}

\title[Limit theorems in the ADRCM]
{Limit theorems under heavy-tailed scenario in the age dependent random connection models}

\author{Christian Hirsch}
\address{Department of Mathematics\\
Aarhus University \\
Ny Munkegade, 118, 8000, Aarhus C, Denmark}
\email{hirsch@math.au.dk}
\author{Takashi Owada}
\address{Department of Statistics\\
Purdue University \\
IN, 47907, USA}
\email{owada@purdue.edu}

\thanks{This research was partially carried out during Owada's visit to Aarhus University funded by AUFF visit grant AUFF-E-2023-6-38. Owada's research was partially supported by the AFOSR grant
  FA9550-22-1-0238 at Purdue University.}

\subjclass[2020]{Primary 60D05, 60G70.  Secondary 60F05, 60G55. }
\keywords{stable limit theorem, extreme value theory, sub-tree count, clique count, scale-free network, random connection model. \vspace{.5ex}}

\begin{abstract}
This paper considers limit theorems associated with subgraph counts in the age-dependent random connection model. 
First, we identify regimes where the count of sub-trees converges weakly to a stable random variable under suitable assumptions on the shape of trees. The proof relies on an  intermediate result on weak convergence of associated point processes towards a Poisson point process. Additionally, we prove the same type of results for the clique counts. Here, a crucial ingredient includes the expectation asymptotics for clique counts, which itself is a result of independent interest.
\end{abstract}

\maketitle
\section{Introduction}
\label{sec:intro}

A central goal in network science is to develop flexible and parsimoniously parameterized random graph models reflecting the key features of a variety of real-world networks. Most prominently, this includes the appearance of hubs exhibiting  atypically large degrees and allowing  short connecting paths between network nodes chosen at random.

One of the most prominent network models in this context  is the preferential attachment model, introduced by Barab\'asi and Albert in \cite{ba}; see also \cite{hofstad:2017} for a comprehensive exposition on this topic. In this model, a network is constructed by adding nodes one after another, and each new node is connected to existing nodes with a probability proportional to their degrees. This model has been shown to capture the scale-free property of many real-world networks, where the degree distribution follows a power law. However, a critical drawback of the standard preferential attachment model is its vanishing clustering coefficient. Specifically, two nodes sharing a common neighbor fail to be connected with a substantially higher probability than  any two arbitrary nodes. This indicates a stark contrast to most  real-world datasets. 

In order to correct such deficit of the preferential attachment model, several approaches have been proposed. A particularly elegant solution is to impose the clustering property by embedding the nodes in Euclidean space. This has led to the \emph{age-dependent random connection model} (ADRCM) and its variants \cite{glm2,komjathy2,glm,komjathy}. The standard random connection model connects any two Poisson points with probability depending only on distance,  leading, therefore, to a light-tailed degree distribution. In contrast, the idea behind the ADRCM is to enforce the scale-free property by endowing the nodes with i.i.d.~weights that influence the connection probability.

Over the past decade, a highly refined machinery was developed for proving central limit theorems (CLTs) associated to Poisson inputs \cite{mal_shot,mehler,yukCLT,rs13}. However, in the setting of complex networks, this theory is often not applicable since the presence of hubs leads to heavy-tailed behavior. Indeed, it is only reasonable to expect a CLT under the assumption of finite variance. However, in many real-world networks, the degree distribution has an infinite second moment. In such situations,  a stable limit theorem naturally arises instead of a CLT.
 
While there exists a large collection of literature on CLTs for Poisson-based network models, there are only a few results for stable limit theorems. For instance, \cite{litvak} is one of such exceptions for a scale-free network model, but their model does not exhibit any geometric components. Moreover, \cite{hj23} establishes a stable limit theorem for the edge count in a class of random connection models. However, this is only of limited use in practice since the number of edges can be fitted with a separate parameter. Hence, there is a natural need for deriving stable limit theorems for more complex test statistics. Finally, it is worth mentioning that \cite{dst} discusses stable limit theorems for a very specific functional given by the sum of inverse powers of pairwise distances in a Poisson point cloud.

Our main results include  point process convergence and stable limit theorems for two more refined test statistics, namely sub-tree counts in Section \ref{sec:tree} and clique counts in Section \ref{sec:clique}. We note that scalings of specific triangle counts in the ADRCM was also considered in \cite{pvh}, but there were no stable limit theorems.

To summarize, the main contributions of this paper are as follows:
\begin{enumerate}
	\item We establish the expectation and variance asymptotics for sub-tree and clique counts with a fixed lowest mark. We note that the missing expectation asymptotics was a major obstacle for the study on clique counts in the closely related hyperbolic model (see, e.g., \cite{subtree}). 
	\item Under suitable assumptions on the  shape of trees, we establish point process convergence and stable limit theorems for the sub-tree counts.
	\item A series of analogous results will be established for the clique counts.
\end{enumerate}
The main idea of the proof in the stable regime is that sub-tree counts are dominated by  a small number of trees with very low marks. The key observation is that the number of sub-trees with a fixed low mark concentrates sharply around its mean. Hence, we can approximate the sub-tree count by a sum of i.i.d.~random variables dependent solely on their marks, allowing us to invoke a stable limit theorem within the framework of extreme value theory, as in \cite{embrechts:kluppelberg:mikosch:1997} and \cite{resnick:2007}.

For the clique counts, the general idea is similar. However, it is substantially more involved to determine the expected number of cliques with a given lowest mark. A similar problem arises in the hyperbolic setting \cite{subtree}, but in the present paper, we shall use a refined analysis  based on the idea that if one of the nodes is connected to a vertex of the lowest mark, then the other vertices are most likely to connect to this vertex as well.

\section{Limit theorems for sub-tree counts}  
\label{sec:tree}

We begin by defining a notion of \emph{graph homomorphisms} to define sub-tree counts of our interest. Suppose $H$ is a directed graph without loops and multiple edges, defined on vertices $[k]:=\{ 1,2,\dots,k \}$ with $E(H)$ being its edge set. Given another directed graph $G=\big(V(G),E(G)\big)$ without loops and multiple edges, let $f:[k]\to V(G)$ denote an \emph{injective} graph homomorphism; that is, $f$ is an injective map and if $(i,j)\in E(H)$ with $i\to j$, then $\big( f(i), f(j) \big)\in E(G)$ with $f(i)\to f(j)$. We then denote by $\mathcal C(H,G)$ the number of injective graph homomorphisms from $H$ to $G$. In other words, this represents the number of copies of $H$ in G, given by 
$$
\mathcal C(H,G) := \sum_{(v_1,\dots,v_k)\in (V(G))_{\neq}^k} \prod_{(i,j)\in E(H), \, i\to j} \one \big\{ (v_i,v_j)\in E(G), \, v_i \to v_j \big\}, 
$$
where $\one \big\{ \cdot \big\}$ is an indicator function and 
$$
\big( V(G) \big)_{\neq}^k := \big\{ (v_1,\dots,v_k)\in V(G)^k: v_i \neq v_j \text{ for } i \neq j \big\}
$$
is a collection of $k$-tuples of distinct elements in $V(G)$. 

Let $\mP$ be a homogeneous Poisson point process on $\bbT := \R\times [0,1]$. The ADRCM, denoted  $\ms{AD}(\beta,\gamma)$, represents a directed graph whose vertex set is $\mP$. 
For the edge set, we set that for $(x,u), (y,v)\in \mP$ with $u\le v$, an edge  is present from $(y,v)$ to $(x,u)$ (we write it as $(y,v)\to (x,u)$) if and only if 
\begin{equation}  \label{e:edge:ADRCM}
|x-y| \le \beta u^{-\gamma} v^{\gamma-1}, 
\end{equation}
where $\beta>0$ and $\gamma\in (0,1)$ are parameters determining the density of edges. For each node $(x,u)$ in $\ms{AD}(\beta,\gamma)$,  $x$ is often called the spatial coordinate and  $u$ is said to be  the time coordinate. 
If $\mP$ is restricted to $\bbT_n := [0,n]\times [0,1]$, $n\ge1$, we denote by $\ms{AD}_n(\beta,\gamma)$ the same ADRCM, but with its vertex set restricted to $\mP_n := \mP \cap \bbT_n$. Note that as in \cite{glm2}, we restrict our attention to the case $\gamma < 1$. Otherwise, every node in $\ms{AD}(\beta,\gamma)$ would have infinite degree with probability 1, which is however a degenerate situation.  Indeed, such degenerate cases were also excluded in \cite{sfPerc}.

Let $\ms{T}$ denote a fixed directed tree on $a+1$ nodes with a  root $\ms{r}$. We assume that all edges emanating from the leaves are directed towards $\ms{r}$. As a consequence,  the out-degree of  $\ms{r}$ becomes zero. Moreover, $\ms{r}$ represents the "lowest mark" vertex, defined as the vertex with the smallest time coordinate when $\ms{T}$ is embedded in $\bbT$; see Figure \ref{fig:tree}.
\begin{figure}
\includegraphics[scale=0.35]{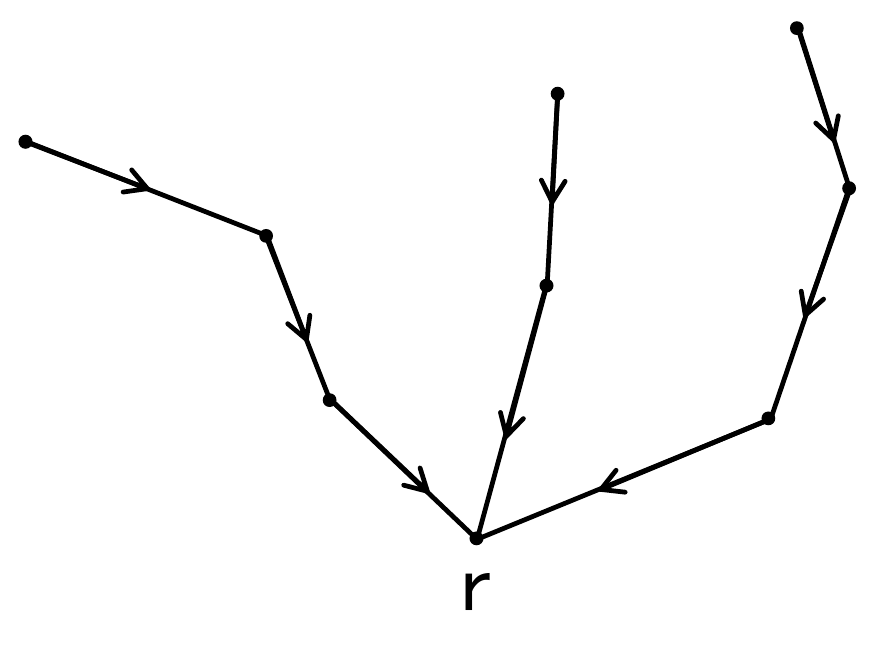}
\caption{\label{fig:tree} \footnotesize{A tree $\ms{T}$ of $9$ nodes and $3$ leaves with $a=8$, where $\ms{r}$ represents a root. }}
\end{figure}
Now, for $(x,u) \in \bbT$, define $\inDt(x,u)$ as the number of injective graph homomorphisms from $\ms{T}$ to $\ms{AD}(\beta,\gamma)$ such that $\ms{r}$ is mapped to $(x,u)$. If $x=0$, we denote this by $\inDt(u) = \inDt(0,u)$. By construction, all the nodes of a tree counted by $\inDt(x,u)$, except for $(x,u)$ itself, have higher marks than $(x,u)$. It should also be noted that $\inDt(x,u)$ does not necessarily enumerate induced trees. Figure \ref{fig:embedding} provides a simple example illustrating $\inDt(x,u)$.
\begin{figure}
\includegraphics[scale=0.35]{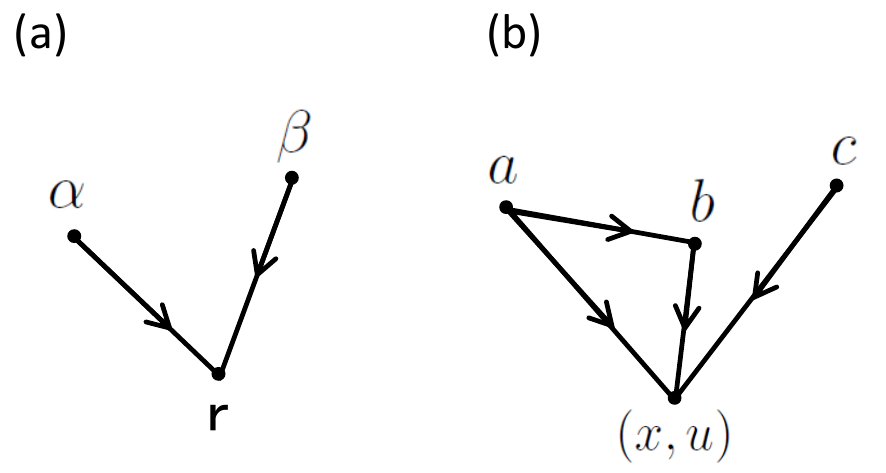}
\caption{\label{fig:embedding} \footnotesize{(a) A tree $\ms{T}$ with $3$ nodes including a root $\ms{r}$. (b) Suppose that four Poisson points, $a$, $b$, $c$, and $(x,u)$ are observed in $\bbT$. Then, $\inDt(x,u)$ counts $6$ injective graph homomorphisms; that is, $f_1(\alpha)=a$, $f_1(\beta)=b$, $f_2(\alpha)=b$, $f_2(\beta)=a$, $f_3(\alpha)=a$, $f_3(\beta)=c$, $f_4(\alpha)=c$, $f_4(\beta)=a$, $f_5(\alpha)=b$, $f_5(\beta)=c$, and $f_6(\alpha)=c$, $f_6(\beta)=b$, while all $f_i$'s  map $\ms{r}$ to $(x,u)$.}}
\end{figure}

One of the  central quantities of interest is the expected value
\begin{equation}  \label{e:def.mu.u}
\mu(u) := \E \big[ \inDt(x,u) \big], \ \ (x,u) \in \bbT. 
\end{equation}
By the stationarity of $\mP$, \eqref{e:def.mu.u} does not depend on the spatial coordinate $x$.
Now, the sub-tree counts of our interest can be defined by 
\begin{equation}  \label{e:def.sub-tree.counts}
 \sum_{P\in \mP_n} \inDt(P), \ \ \ n\ge1. 
\end{equation}
Additionally, we will examine the behavior of the associated point process 
\begin{equation}  \label{e:def.pp}
\sum_{P\in \mP_n} \delta_{a_n^{-1} \inDt(P)}, \ \ \ n\ge1, 
\end{equation}
where $\delta$ denotes the Dirac measure and $a_n$ is a scaling constant defined as 
$$
a_n := \mu\Big( \frac{1}{n} \Big), \ \ \ n \ge 1.
$$
The primary conclusion of our finding is that the behavior of \eqref{e:def.sub-tree.counts} and \eqref{e:def.pp} largely depends on the values of $\gamma$ in \eqref{e:edge:ADRCM}. If $\gamma$ is large enough, i.e., sufficiently close to $1$, then due to \eqref{e:edge:ADRCM}, the root $P=(X,U)\in \mP$ counted by \eqref{e:def.sub-tree.counts} and \eqref{e:def.pp}, tends to become a "hub" node, especially when $U$ is small enough. This hub node is more likely to connect with many other nodes of higher marks, thereby dominating the sub-tree counts in \eqref{e:def.sub-tree.counts}. As a result, with proper normalization, \eqref{e:def.sub-tree.counts} converges weakly to a stable distribution. Conversely, when $\gamma$ is near zero, the influence of such hub nodes decreases significantly, leading the sub-tree counts \eqref{e:def.sub-tree.counts} to follow a standard CLT with Gaussian limits, as will be detailed in a follow-up work for certain tree types \cite{aa}.

We start by exploring the asymptotics of the expectation  $\mu(u)$.
 The proof of Proposition \ref{p:exp.asym} below is deferred to the Appendix in Section \ref{sec:proofs}. Throughout the paper, denote by $C$ a generic positive constant, which varies between (and even within) lines and is independent of $u$ or $n$ depending on the context. 
 
Given a tree $\ms{T}$ as above, let $v_1,\dots,v_\ell$ be the leaves with no further children. For $i=1,\dots,\ell$, let $m_i$ be the number of nodes between $v_i$ and the nearest intersection node, excluding both $v_i$ and the intersection node themselves.  If the path beginning at $v_i$ does not merge with other paths until it reaches the root $\ms{r}$, then $m_i$ is defined as the number of vertices between $v_i$ and $\ms{r}$ without counting $v_i$ and $\ms{r}$ themselves. See Figure \ref{fig:Y-shape.and.non.Y-shape} (a) for an illustration of a tree of $5$ leaves. 
 
\begin{proposition}  \label{p:exp.asym}
Given a tree $\ms{T}$ as above, we have, as $u\to 0$, 
\begin{equation}  \label{e:exp.asym}
\mu(u) \sim C u^{-\ell\gamma} (\log u^{-1})^m, 
\end{equation}
where  $m:=\sum_{i=1}^\ell m_i$. Furthermore, $\mu(u)$ is increasing as $u\downarrow 0$. 
\end{proposition}

\begin{figure}
\includegraphics[scale=0.35]{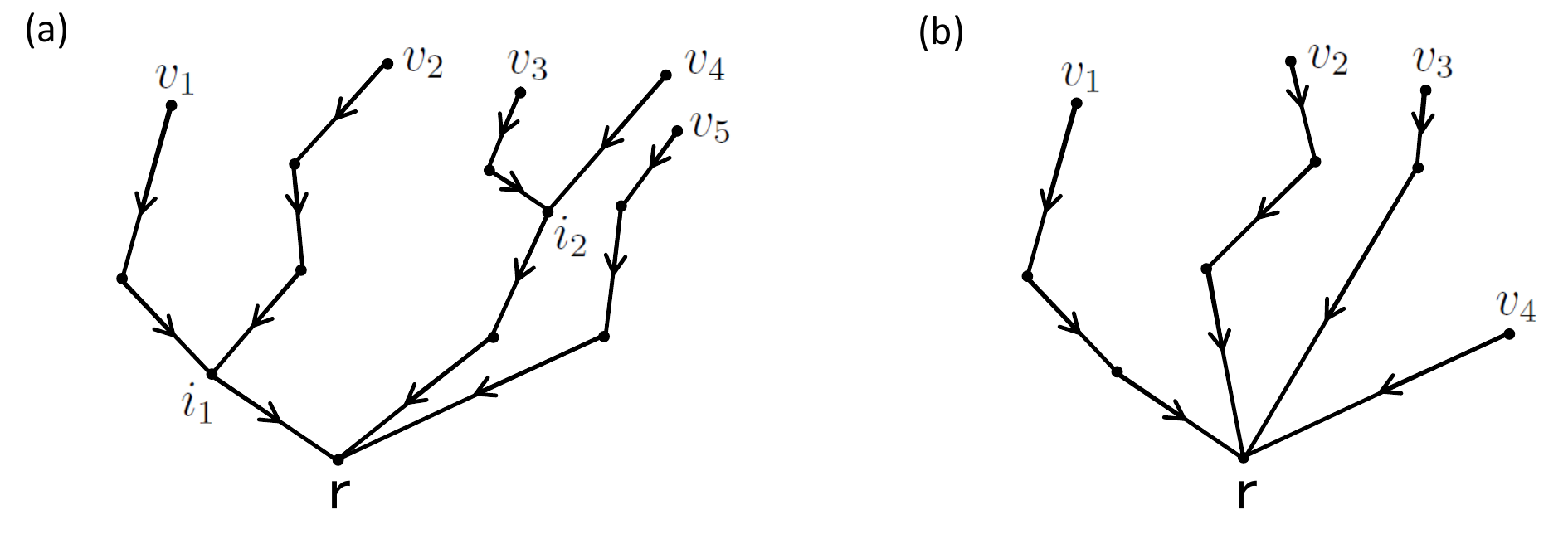}
\caption{\label{fig:Y-shape.and.non.Y-shape} \footnotesize{(a) A tree of $5$ leaves $v_1, v_2, v_3$, $v_4$, and $v_5$, where $i_1$ and $i_2$ are the nearest intersection nodes. In this case, we have $m_1=1$, $m_2=2$, $m_3=1$, $m_4=0$, and $m_5=2$. (b) A tree of $4$ leaves $v_1, v_2, v_3$, and $v_4$. This tree does not hold intersection nodes such as $i_1$, $i_2$ in  Case (a). Here we see that  $m_1=m_2=2$, $m_3=1$, and $m_4=0$.}}
\end{figure}

Proposition \ref{p:exp.asym} indicates that $\mu(u)$ is a regularly varying function at the origin with exponent $-\ell\gamma$. We denote this by $\mu\in \ms{RV}_{-\ell\gamma}^{(0)}$. 

In contrast to the general tree structures   in Proposition \ref{p:exp.asym}, we now start focusing on more specific trees. More concretely, we consider a tree with $\ell$ leaves $v_1, \dots, v_\ell$ with no further children, such that for each $i=1, \dots, \ell$, the path beginning at $v_i$ does not merge with other paths beginning at different $v_j$'s, $j \neq i$, before reaching the root $\ms{r}$. Assume further that for each $i=1, \dots, \ell$, there are exactly $m_i$ nodes between $v_i$ and $\ms{r}$ (excluding $v_i$ and $\ms{r}$); see Figure \ref{fig:Y-shape.and.non.Y-shape} (b).  Finally, define again $m := \sum_{i=1}^\ell m_i$. The proof of Proposition \ref{p:var.asym.specific} below is presented in the Appendix.

%\begin{figure}[h]
%\begin{tikzpicture}
  % Root node
%  \node[fill, circle] (root) at (0,0) {};
  
  % Branch 1 nodes
%  \node[fill, circle] (branch1a) at (-2,2) {};
%  \node[fill, circle] (branch1b) at (-2,4) {};
%	\coordinate[label=180:{$v_1$}] (A)  at (-2.2, 4);
  
  % Branch 2 nodes
%  \node[fill, circle] (branch2a) at (0,2) {};
%  \node[fill, circle] (branch2b) at (0,4) {};
%	\coordinate[label=180:{$v_2$}] (A)  at (-0.2, 4);
  
  % Branch 3 nodes
%  \node[fill, circle] (branch3a) at (2,2) {};
%  \node[fill, circle] (branch3b) at (2,4) {};
%	\coordinate[label=180:{$v_3$}] (A)  at (1.8, 4);
  
  % Edges
%  \draw (root) -- (branch1a);
%  \draw (branch1a) -- (branch1b);
  
%  \draw (root) -- (branch2a);
%  \draw (branch2a) -- (branch2b);
  
%  \draw (root) -- (branch3a);
%  \draw (branch3a) -- (branch3b);
%\end{tikzpicture}
%	\caption{Illustration of tree type in Proposition \ref{p:exp.asym}}
%	\label{fig:a}
%\end{figure}

\begin{proposition}  \label{p:var.asym.specific}
For the tree $\ms{T}$ defined in the last paragraph (see also Figure \ref{fig:Y-shape.and.non.Y-shape} (b)), there exists a finite constant $C>0$ such that  
\begin{equation}  \label{e:var.asym.specific}
\text{Var}\big( \inDt(u) \big) \le \begin{cases}
Cu^{-2\ell\gamma} (\log u^{-1})^{2m-1} & \text{if } m\ge1, \\[5pt]
Cu^{-(2\ell-1)\gamma} & \text{if } m=0, 
\end{cases}
\ \ \ u\in (0,1). 
\end{equation}
\end{proposition}

\begin{remark}
We stress that the variance asymptotics in Proposition \ref{p:var.asym.specific} imposes a stricter constraint than  Proposition  \ref{p:exp.asym}. In our conjecture, this restriction is not merely an artifact of the proof but suggests something more fundamental: the presence of intersection nodes as in Figure \ref{fig:Y-shape.and.non.Y-shape} (a), may result in a variance rate that differs from those in Proposition \ref{p:var.asym.specific}.
\end{remark}

We now recall  that stable limit theorems arise only when $\gamma$ is sufficiently large. Specifically, for the case of sub-tree counts, we need to require $\gamma > 1/(2\ell)$.
It follows from Proposition \ref{p:exp.asym} that 
$$
a_n \sim Cn^{\ell\gamma} (\log n)^m,  \ \ \text{as } n\to\infty, 
$$
with the same constant $C>0$ as in \eqref{e:exp.asym}, and hence, $(a_n)_{n\ge1}$ is a regularly varying sequence at infinity with exponent $\ell\gamma$. We denote this as $(a_n)_{n\ge1} \in \ms{RV}_{\ell\gamma}^{(\infty)}$. Finally, since $\mu$ is increasing as $u\downarrow0$, the general inverse of $\mu$, defined by 
$$
\muti(x) := \inf \big\{ s\ge0:  \mu (s) \le x\big\}, 
$$
is also regularly varying with index $-1/(\ell\gamma)$; that is,  $\muti \in \ms{RV}_{-1/(\ell\gamma)}^{(\infty)}$ (see  Proposition 2.6 in \cite{resnick:2007}). 

Our first  primary result focuses on the asymptotic behavior of the point process in \eqref{e:def.pp}. Before presenting this result, it is important to note that all subsequent results in this section are specifically applicable to trees as outlined in Proposition \ref{p:var.asym.specific}. 

Let $M_p\big( (0,\infty] \big)$ denote  the space of point measures on $(0,\infty]$. Henceforth, we write $\ms{PPP}(\xi)$ for a Poisson point process with a given intensity measure $\xi$. 

\begin{theorem}[Point process convergence for sub-tree counts]  \label{t:pp.weak.conv}
Let $1/(2\ell) < \gamma <1$. Then, in the space $M_p\big( (0,\infty] \big)$, the point process \eqref{e:def.pp}  converges weakly to a Poisson point process on $(0,\infty]$ with intensity measure $\kappa_{\ell\gamma}$, such that $\kappa_{\ell\gamma}\big((y,\infty]\big) = y^{-1/(\ell\gamma)}$, $y>0$. 
\end{theorem}

\begin{proof}

We first demonstrate that the process  $\sum_{P=(X,U)\in \mathcal P_n}\delta_{a_n^{-1}\mu(U)}$ converges weakly to $\ms{PPP}(\kappa_{\ell\gamma})$. To this end, it suffices to show that as $n\to\infty$, 
\begin{align}
&\sum_{P=(X,U)\in \mathcal P_n}\delta_{a_n^{-1}\mu(U)} - \sum_{i=1}^n \delta_{a_n^{-1}\mu(U_i)} \stackrel{p}{\to} \emptyset \ \ \text{in } M_p\big((0,\infty]\big),  \label{e:diff.de.Poissonization} \\
&\sum_{i=1}^n \delta_{a_n^{-1}\mu(U_i)} \Rightarrow \ms{PPP}(\kappa_{\ell\gamma}), \ \ \text{in }  M_p\big((0,\infty]\big), \label{e:de.Poissonization.weak.conv}
\end{align}
where $\emptyset$ is the null measure (i.e., the measure assigning zeros to all measurable sets) and $U_1,\dots,U_n$ are i.i.d.~uniform random variables over $(0,1)$. 
For the proof of \eqref{e:diff.de.Poissonization}, it suffices to show that 
$$
\sum_{P=(X,U)\in \mathcal P_n} f \big( a_n^{-1}\mu(U) \big)  -\sum_{i=1}^n f \big( a_n^{-1}\mu(U_i) \big)  \stackrel{p}{\to} 0, 
$$
for every $f\in C_K^+\big((0,\infty]\big)$, where $C_K^+\big((0,\infty]\big)$ represents a family of non-negative and continuous functions on $(0,\infty]$ with compact support. Since $f$ has compact support, there exists $\delta_0>0$ so that $\ms{supp}(f)\subset [\delta_0,\infty]$ ($\ms{supp}(f)$ denotes the support of $f$). It then follows from the regular variation of $\muti$  that as $n\to\infty$, 
\begin{align}
\begin{split}  \label{e:de-Poi.approx}
&\E  \bigg[ \bigg|\sum_{P=(X,U)\in \mP _n} f \big( a_n^{-1}\mu(U) \big) -\sum_{i=1}^n f \big( a_n^{-1}\mu(U_i) \big)  \bigg| \bigg] \le \E \big[ f(a_n^{-1}\mu(U_1)) \big] \E \big[ |\mP(\bbT_n)-n| \big] \\
&\le \| f \|_\infty \P \big( a_n^{-1}\mu(U_1) \ge \delta_0 \big)  \E \big[ |\mP(\bbT_n)-n| \big] \\
&\le C\|f\|_\infty \muti(a_n\delta_0) \sqrt{n} \le C\|f\|_\infty \delta_0^{-1/(\ell\gamma)} n^{-1/2}\to0, \ \ \ n\to\infty,
\end{split} 
\end{align}
thereby yielding \eqref{e:diff.de.Poissonization}. 

Next, according to Theorem 6.2 and Lemma 6.1 in \cite{resnick:2007}, \eqref{e:de.Poissonization.weak.conv} is equivalent to 
$$
n\P\big( a_n^{-1}\mu (U_1)>y \big) \to \kappa_{\ell\gamma}\big((y,\infty]\big) = y^{-1/(\ell\gamma)}, 
$$
for every $y>0$, which however follows from the regular variation of $\muti$. 

Now, one can complete the proof, if it can be shown that as $n\to\infty$, 
$$
\sum_{P\in \mP_n} \delta_{a_n^{-1}\inDt(P)} - \sum_{P=(X,U)\in \mP_n} \delta_{a_n^{-1}\mu(U)} \stackrel{p}{\to} \emptyset, \ \ \text{in } M_p\big( (0,\infty] \big); 
$$
equivalently, for every $f\in C_K^+\big( (0,\infty] \big)$, 
$$
\sum_{P\in \mP_n} f\big(  a_n^{-1}\inDt(P)\big)  - \sum_{P=(X,U)\in \mP_n} f\big(  a_n^{-1}\mu(U)\big) \stackrel{p}{\to}0. 
$$
Recall that $\delta_0>0$ satisfies $\ms{supp}(f)\subset [\delta_0,\infty]$ and fix $\delta'\in (0,\delta_0/2)$. For every $\delta\in (0,\delta_0/2)$, it follows from the Mecke formula  for Poisson point processes (see, e.g., \cite[Theorem 4.4]{LastPenroseBook}) that 
\begin{align*}
&\E \bigg[ \bigg|  \sum_{P=(X,U)\in \mP_n}f\big(  a_n^{-1}\inDt(P)\big)  - f\big(  a_n^{-1}\mu(U)\big) \bigg| \bigg]  \\
&\le n\int_0^1 \E \Big[ \Big|  f\big(  a_n^{-1}\inDt(u)\big)  - f\big(  a_n^{-1}\mu(u)\big)\Big|\Big]\dif u \\
&\le 2\| f \|_\infty n\int_0^1 \P \Big( \big| \inDt(u)  -  \mu(u)  \big| >\delta a_n\Big)\dif u \\
&\quad + n\int_0^1 \E \bigg[ \Big|  f\big(  a_n^{-1}\inDt(u)\big)  - f\big(  a_n^{-1}\mu(u)\big)\Big| \, \one \Big\{ \big| \inDt(u)  - \mu(u)  \big|\le \delta a_n \Big\}  \bigg]\dif u \\
&=:A_n + B_n. 
\end{align*}

We first deal with $B_n$. Note that we may insert an additional indicator $\one\{\mu(u) > \delta' a_n\}$. Indeed, suppose $\mu(u) \le \delta' a_n$; then $f(a_n^{-1}\mu(u)) = 0$ as $a_n^{-1}\mu(u)$ lies outside of  $\ms{supp}(f)$. Moreover, by the triangle inequality, $a_n^{-1}\inDt(u) \le \delta + \delta' < \delta_0$; hence, $f(a_n^{-1}\inDt(u)) = 0$. Thus, one can add the indicator $\one\{\mu(u) > \delta' a_n\}$ to $B_n$, such  that 
\begin{align*}
B_n &= n\int_0^1 \E \bigg[ \Big|  f\big(  a_n^{-1}\inDt(u)\big)  - f\big(  a_n^{-1}\mu(u)\big)\Big| \, \\
&\qquad \qquad \times \one \Big\{ \big| \inDt(u)  - \mu(u)  \big|\le \delta a_n, \, \mu(u)>\delta' a_n \Big\}  \bigg]\dif u \\
&\le \omega_f(\delta)n\int_0^1 \one \big\{ \mu(u)>\delta'a_n \big\} \dif u = \omega_f(\delta) n\muti(\delta'a_n), 
\end{align*}
where $\omega_f$ is the modulus of continuity of $f$. The last term  converges to $(\delta')^{-1/(\ell\gamma)}\omega_f(\delta)$ as $n\to\infty$, which clearly tends to $0$ as $\delta\downarrow0$, since $f$ is uniformly continuous. 

It is now sufficient to show that $A_n\to0$ as $n\to\infty$ for every $\delta>0$. By Chebyshev's inequality, we only need to demonstrate that 
\begin{equation}  \label{e:A_n.conv}
\frac{n}{a_n^2} \int_{\vep_1 n^{-1}}^1 \text{Var}\big( \inDt(u) \big) \dif u \to 0, \ \ \text{as } n\to\infty, 
\end{equation}
for every $\vep_1>0$. 
If $m\ge1$, we obtain from Proposition \ref{p:var.asym.specific} that 
\begin{align*}
\frac{n}{a_n^2} \int_{\vep_1n^{-1}}^1 \text{Var}\big( \inDt(u) \big) \dif u &\le \frac{Cn}{a_n^2}  \int_{\vep_1n^{-1}}^1 u^{-2\ell\gamma} (\log u^{-1})^{2m-1} \dif u \le \frac{Cn}{a_n^2} \int_0^{\vep_1^{-1}n} v^{2\ell\gamma-2} (\log v)^{2m-1}\dif v. 
\end{align*}
Since $v^{2\ell\gamma-2} (\log v)^{2m-1} \in \ms{RV}_{2\ell\gamma-2}^{(\infty)}$ with $2\ell\gamma-2>-1$, Karamata's theorem (see \citep[Theorem 2.1]{resnick:2007}) yields that as $n\to\infty$, 
\begin{align*}
\frac{n}{a_n^2} \int_0^{\vep_1^{-1}n} v^{2\ell\gamma-2} (\log v)^{2m-1}\dif v &\sim \frac{n}{a_n^2}\cdot \frac{\vep_1^{-1}n}{2\ell\gamma-1}\, (\vep_1^{-1}n)^{2\ell\gamma-2} \big( \log (\vep_1^{-1}n) \big)^{2m-1} \\
&\sim C\, \frac{(\log n -\log \vep_1)^{2m-1}}{(\log n)^{2m}} \to 0, \ \ \ n\to\infty. 
\end{align*}
The last convergence is assured by the restriction $m\ge1$. 

On the other hand, if $m=0$, Proposition \ref{p:var.asym.specific} gives that 
\begin{equation}  \label{e:var.lower.epsilon1}
\frac{n}{a_n^2} \int_{\vep_1n^{-1}}^1 \text{Var}\big( \inDt(u) \big) \dif u \le Cn^{1-2\ell\gamma} \int_{\vep_1n^{-1}}^1 u^{-(2\ell-1)\gamma} \dif u. 
\end{equation}
If $1/(2\ell) < \gamma \le 1/(2\ell-1)$, the last term is upper bounded by $Cn^{1-2\ell\gamma} \log n \to 0$ as $n\to\infty$, where the logarithmic factor is needed only when $\gamma=1/(2\ell-1)$. If $1/(2\ell-1)<\gamma<1$, \eqref{e:var.lower.epsilon1} is bounded by $Cn^{-\gamma}\to 0$ as $n\to\infty$. 
\end{proof}

Theorem \ref{t:pp.weak.conv} immediately yields weak convergence for the maximum of $\big(\inDt(P), \, P \in \mP_n \big)$. 

\begin{corollary}  \label{cor:maxima}
Let $1/(2\ell) <\gamma <1$. Then, 
we have, as $n\to\infty$, 
$$
\frac{1}{a_n}\bigvee_{P\in \mP_n} \inDt(P) \Rightarrow Y_{\ell\gamma}, 
$$
where $\P(Y_{\ell\gamma} \le y) =e^{-\kappa_{\ell\gamma}((y,\infty))} = e^{-y^{-1/(\ell\gamma)}}$, i.e., $Y_{\ell\gamma}$ is a $1/(\ell\gamma)$-Fr\'echet random variable. 
\end{corollary}
\begin{proof}
Use $\sum_\ell \delta_{j_\ell}$ to represent the Poisson point process $\ms{PPP}(\kappa_{\ell\gamma})$. It follows from Theorem \ref{t:pp.weak.conv} that for every $y>0$, 
\begin{align*}
\P \Big( \frac{1}{a_n}\bigvee_{P\in \mP_n} \inDt(P) \le y \Big) &= \P\Big(\sum_{P\in \mP_n}\delta_{a_n^{-1}\inDt(P)} \big((y,\infty)\big) = 0  \Big)\\
&\to \P \Big( \sum_\ell \delta_{j_\ell}\big((y,\infty)\big)=0 \Big) = e^{-\kappa_{\ell\gamma}((y,\infty))} = \P(Y_{\ell\gamma}\le y). 
\end{align*}
\end{proof}

Our next goal is to establish the stable limit theorem for the sub-tree counts  in \eqref{e:def.sub-tree.counts}. In contrast to Theorem \ref{t:pp.weak.conv}, we observe a phase transition based on the values of $\gamma$. Specifically, when $1/(2\ell) < \gamma < 1/\ell$, the sub-tree counts \eqref{e:def.sub-tree.counts} must be centered by its expectation, with the weak limit being a zero-mean $1/(\ell\gamma)$-stable random variable. Conversely, if $\ell>1$ and $1/\ell<\gamma < 1$, centering is unnecessary, and the weak limit is a $1/(\ell\gamma)$-stable subordinator, which is a $1/(\ell\gamma)$-stable random variable whose support is restricted to the positive half-line (see \cite{samorodnitsky:taqqu:1994}). Although the same limit theorem is expected to hold in the boundary case $\gamma = 1/\ell$, it requires a separate argument with an appropriate truncation of the centering terms. To avoid technical complexities, we have chosen not to address this boundary case in the present paper.

\begin{theorem}[Stable limit theorem for sub-tree counts]  \label{t:stable.limit}
$(i)$ If $1/(2\ell) < \gamma < 1/\ell$, then 
$$
\frac{1}{a_n}\, \Big( \sum_{P\in \mP_n} \inDt(P) - \E \Big[  \sum_{P\in \mP_n} \inDt(P)  \Big] \Big) \Rightarrow S_{1/(\ell\gamma)}, 
$$
where $S_{1/(\ell\gamma)}$ is a zero-mean $1/(\ell\gamma)$-stable random variable such that 
\begin{equation}  \label{e:ch.f.stable}
\E \big[ e^{iv S_{1/(\ell\gamma)}} \big] = \exp \Big\{ \int_0^{\infty} \big( e^{iv x} -1- iv x \big) \kappa_{\ell\gamma}(\dif x) \Big\}, \ \ \ v\in \R. 
\end{equation}
(ii) If $\ell>1$ and $1/\ell<\gamma<1$, then 
$$
\frac{1}{a_n}  \sum_{P\in \mP_n} \inDt(P) \Rightarrow \widetilde S_{1/(\ell\gamma)}, 
$$
where $\widetilde S_{1/(\ell\gamma)}$ is a $1/(\ell\gamma)$-stable subordinator such that 
$$
\E \big[ e^{iv \widetilde S_{1/(\ell\gamma)}} \big] = \exp \Big\{ \int_0^{\infty} \big( e^{iv x} -1 \big) \kappa_{\ell\gamma}(\dif x) \Big\}, \ \ \ v\in \R. 
$$
\end{theorem}

The first step in proving Theorem \ref{t:stable.limit} is to establish analogous results on weak convergence for the version obtained by replacing $\inDt(P)$ in Theorem \ref{t:stable.limit} with $\mu(U)$, where $P = (X,U) \in \mP$.

\begin{proposition}  \label{p:stable.limit.uniform}
$(i)$ If $1/(2\ell) < \gamma < 1/\ell$, then 
$$
\frac{1}{a_n}\, \Big( \sum_{P=(X,U)\in \mP_n} \mu(U) - \E \Big[  \sum_{P=(X,U)\in \mP_n} \mu(U) \Big] \Big) \Rightarrow S_{1/(\ell\gamma)}. 
$$
(ii) If $\ell>1$ and $1/\ell<\gamma<1$, then 
$$
\frac{1}{a_n}  \sum_{P=(X,U)\in \mP_n} \mu(U)\Rightarrow \widetilde S_{1/(\ell\gamma)}. 
$$
\end{proposition}

\begin{proof}[Proof of Proposition \ref{p:stable.limit.uniform} $(i)$]
Given $\vep>0$, let $T_\vep: M_p\big( (0,\infty] \big) \to \R_+$ be the map $T_\vep \big( \sum_\ell \delta_{z_\ell} \big) = \sum_\ell z_\ell\one\{ z_\ell\ge\vep \}$. Note that $T_\vep$ is almost surely continuous with respect to the probability law of $\ms{PPP}(\kappa_{\ell\gamma})$; see Section 7.2.3 of \cite{resnick:2007} for more details. As in the proof of Corollary \ref{cor:maxima}, we use $\sum_\ell\delta_{j_\ell}$ to represent $\ms{PPP}(\kappa_{\ell\gamma})$. 
It follows from \eqref{e:diff.de.Poissonization} and \eqref{e:de.Poissonization.weak.conv} that as $n\to\infty$, 
\begin{equation}  \label{e:weak.Poisson.point.U}
\sum_{P=(X,U)\in \mP _n} \delta_{a_n^{-1}\mu(U)} \Rightarrow \ms{PPP}(\kappa_{\ell\gamma}), \ \ \text{in } M_p\big( (0,\infty] \big). 
\end{equation}
Applying the continuous mapping theorem to  \eqref{e:weak.Poisson.point.U},  
\begin{equation}  \label{e:cont.mapping.thm}
\frac{1}{a_n}\sum_{P=(X,U)\in \mP_n}\mu (U)\one \{ \mu(U)\ge \vep a_n  \} \Rightarrow \sum_\ell j_\ell \one \{ j_\ell\ge\vep \}, \  \ \text{as } n\to\infty. 
\end{equation}
Since $\gamma< 1/\ell$, it also follows that 
$$
\frac{1}{a_n}\E\Big[ \sum_{P=(X,U)\in \mP_n}\mu (U)\one \{ \mu(U)\ge \vep a_n  \}\Big] \to \int_{\vep}^\infty x \kappa_{\ell\gamma}(\dif x), \ \ \ n\to\infty, 
$$
and thus, as $n\to\infty$, 
\begin{align*}
&\frac{1}{a_n} \bigg(\sum_{P=(X,U)\in \mP_n}\mu (U)\one \{ \mu(U)\ge \vep a_n \} -  \E\Big[ \sum_{P=(X,U)\in \mP_n}\mu (U)\one \{  \mu(U)\ge \vep a_n  \}\Big] \bigg) \\
&\qquad \qquad \Rightarrow \sum_\ell j_\ell \one \{ j_\ell\ge\vep \} - \int_{\vep}^\infty  x \kappa_{\ell \gamma}(\dif x). 
\end{align*}
Moreover, 
$$
\sum_\ell j_\ell \one \{ j_\ell\ge\vep \} - \int_{\vep}^\infty  x \kappa_{\ell \gamma}(\dif x) \Rightarrow S_{1/(\ell \gamma)}, \ \ \text{as } \vep \to 0; 
$$
see Section 5.5 in \cite{resnick:2007}. 
To finish the proof, we need to show that for every $\eta>0$, 
\begin{align*}
&\limsup_{\vep\to0} \limsup_{n\to\infty} \P \bigg( \, \bigg| \sum_{P=(X,U)\in \mP_n}\mu (U)\one \{ \mu(U)\le  \vep a_n \} \\
&\qquad \qquad\qquad \qquad -  \E\Big[ \sum_{P=(X,U)\in \mP_n}\mu (U)\one \{  \mu(U)\le \vep a_n  \}\Big] \bigg| > \eta a_n\bigg)=0
\end{align*}
	(see Theorem 3.5 in \cite{resnick:2007}). By Chebyshev's inequality and the Mecke formula for Poisson point processes, one can bound the last term by 
\begin{align*}
&\frac{1}{\eta^2}\, \limsup_{\vep\to0} \limsup_{n\to\infty} \frac{1}{a_n^2}\, \text{Var} \Big( \sum_{P=(X,U)\in \mP_n}\mu (U)\one \{ \mu(U)\le  \vep a_n \} \Big)  \\
&=\frac{1}{\eta^2}\, \limsup_{\vep\to0} \limsup_{n\to\infty} \frac{n}{a_n^2}\,\int_0^1 \mu(u)^2 \one \{ \mu(u)\le \vep a_n \}\dif u \\
&\le \frac{1}{\eta^2}\, \limsup_{\vep\to0} \limsup_{n\to\infty} \frac{n}{a_n^2}\,\int_0^{1/\muti(\vep a_n)} \mu (v^{-1})^2 v^{-2}\dif v. 
\end{align*}
Since $\mu(v^{-1})^2 v^{-2}\in \ms{RV}_{2\ell\gamma-2}^{(\infty)}$ with $2\ell\gamma-2>-1$,  Karamata's theorem gives that 
\begin{align*}
\frac{n}{a_n^2}\,\int_0^{1/\muti(\vep a_n)} \mu (v^{-1})^2 v^{-2}\dif v &\sim \frac{n}{a_n^2}\cdot \frac{1}{2\ell\gamma-1}\cdot \frac{1}{\muti(\vep a_n)}\, \mu \big( \muti(\vep a_n) \big)^2 \muti(\vep a_n)^2 \\
&\to \frac{\vep^{2-1/(\ell\gamma)}}{2\ell\gamma-1}, \ \ \ n\to\infty, \\
&\to 0, \ \  \text{as } \vep \to 0,  
\end{align*}
as desired. 
\end{proof}

\begin{proof}[Proof of Proposition \ref{p:stable.limit.uniform} $(ii)$]
Repeating the same argument, \eqref{e:cont.mapping.thm} follows as before. 
Since $\gamma>1/\ell$, it directly follows that 
$$
\sum_\ell j_\ell \one \{ j_\ell\ge\vep \}  \Rightarrow \widetilde S_{1/(\ell \gamma)}, \ \ \text{as } \vep \to  0; 
$$
see again Section 5.5 in \cite{resnick:2007}. 
To finish the proof, we only need to show that for every $\eta>0$, 
\begin{align*}
&\limsup_{\vep\to0} \limsup_{n\to\infty} \P \bigg( \sum_{P=(X,U)\in \mP_n}\mu (U)\one \{ \mu(U)\le  \vep a_n \}  > \eta a_n\bigg)=0. 
\end{align*}
By Markov's inequality and the Mecke formula, the last term above is upper bounded by 
\begin{align*}
&\frac{1}{\eta}\, \limsup_{\vep\to0} \limsup_{n\to\infty} \frac{1}{a_n}\, \E \Big[ \sum_{P=(X,U)\in \mP_n}\mu (U)\one \{ \mu(U)\le  \vep a_n \} \Big] \\
	&\le \frac{1}{\eta}\, \limsup_{\vep\to0} \limsup_{n\to\infty} \frac{n}{a_n}\,\int_0^{1/\muti(\vep a_n)} \mu (v^{-1}) v^{-2}\dif v. 
\end{align*}
Since $\mu(v^{-1}) v^{-2}\in \ms{RV}_{\ell\gamma-2}^{(\infty)}$ with $\ell\gamma-2>-1$, Karamata's theorem concludes that 
$$
\frac{n}{a_n}\,\int_0^{1/\muti(\vep a_n)} \mu (v^{-1}) v^{-2}\dif v 
\sim \frac{\vep^{1-1/(\ell\gamma)}}{\ell\gamma-1}, \ \ \text{as } n\to\infty. 
$$
The last term converges to $0$ as $\vep\to0$. 
\end{proof}

For the proof of Theorem \ref{t:stable.limit}, we start with statement $(ii)$. Henceforth, we put $\mP_{n, a}^{(\uparrow)} := \mP \cap ([0, n] \times [a, 1])$, and  $\mP_{n,  a}^{(\downarrow)}  := \mP \cap ([0, n] \times [0, a])$ for $a\in [0,1]$. 

\begin{proof}[Proof of Theorem \ref{t:stable.limit} $(ii)$]
By Proposition \ref{p:stable.limit.uniform} $(ii)$, it suffices to verify that 
\begin{align}
&\limsup_{M\to\infty}\limsup_{n\to\infty} \P \bigg( \sum_{P=(X,U)\in \mP_{n,  M/n}^{(\uparrow)}} \mu(U)  > \vep a_n \bigg)=0, \ \ \text{for every } \vep>0,  \label{e:cond1.stable.limit}  \\
&\limsup_{M\to\infty}\limsup_{n\to\infty} \P \bigg( \sum_{P=(X,U)\in \mP_{n,  M/n}^{(\uparrow)} } \inDt (P) > \vep a_n \bigg)=0, \ \ \text{for every } \vep>0,  \label{e:cond2.stable.limit} \\
&\limsup_{n\to\infty} \P \bigg( \, \Big|\sum_{P=(X,U)\in \mP_{n, M/n}^{(\downarrow)} } \hspace{-15pt}\big( \inDt (P) - \mu(U)\big) \Big| > \vep a_n \bigg)=0, \ \ \text{for every } \vep, M >0.  \label{e:cond3.stable.limit}
\end{align}
For the proof of \eqref{e:cond1.stable.limit}, it follows from Markov's inequality  that
\begin{align}
\begin{split}   \label{e:Karamata.cond1}
\P \bigg( \sum_{P=(X,U)\in \mP_{n,  M/n}^{(\uparrow)} } \mu(U)   > \vep a_n \bigg) 
&\le \frac{1}{\vep a_n}\, \E \Big[ \sum_{P=(X,U)\in \mP_{n,  M/n}^{(\uparrow)} } \mu(U) \Big] \\
&\le \frac{n}{\vep a_n}\int_0^{M^{-1}n}\mu(v^{-1})v^{-2}\dif v. 
\end{split}
\end{align} 
Since $\mu(v^{-1})v^{-2}\in \ms{RV}_{\ell\gamma-2}^{(\infty)}$ with $\ell\gamma -2>-1$, Karamata's theorem concludes  that as $n\to\infty$, 
$$
\frac{n}{a_n}\int_0^{M^{-1}n}\mu(v^{-1})v^{-2}\dif v \sim \frac{n}{a_n} \cdot \frac{M^{-1}n\mu(Mn^{-1})}{\ell\gamma-1}\cdot  \frac{1}{(M^{-1}n)^2} \to CM^{1-\ell\gamma}. 
$$
The last term vanishes as $M\to\infty$. 
\medskip

For the proof of  \eqref{e:cond2.stable.limit}.  Markov's inequality yields that 
\begin{align*}
\P \bigg( \sum_{P\in \mP_{n, M/n}^{(\uparrow)} } \inDt(P)   > \vep a_n \bigg) 
&\le \frac{1}{\vep a_n}\, \E \Big[ \sum_{P\in \mP_{n,  M/n}^{(\uparrow)} } \inDt(P) \Big] \\
&\le \frac{n}{\vep a_n}\int_0^{M^{-1}n}\mu(v^{-1})v^{-2}\dif v; 
\end{align*} 
we have  obtained the same bound as \eqref{e:Karamata.cond1}. 
\medskip

Finally, we turn our attention to \eqref{e:cond3.stable.limit}. We first have 
\begin{align*}
&\P \bigg(\sum_{P=(X,U)\in \mP_{n,  M/n}^{(\downarrow)} } \big|  \inDt(P)-\mu(U) \big| >\vep a_n \bigg) \\
&\le \P \Big( \bigcup_{P=(X,U)\in \mP_{n,  M/n}^{(\downarrow)} } \Big\{ \big| \inDt(P)-\mu(U) \big|>\eta a_n \Big\} \Big) \\
&+\P \bigg( \sum_{P=(X,U)\in \mP_{n,  M/n}^{(\downarrow)} } \big|  \inDt(P)-\mu(U) \big|\, \one \Big\{ \big|  \inDt(P)-\mu(U) \big|\le \eta a_n  \Big\} >\vep a_n \bigg) \\
&=:A_n+B_n, 
\end{align*}
where $0<\eta<\vep$. By Markov's inequality, 
\begin{align*}
	B_n &\le \frac{1}{\vep a_n}\, \E \bigg[ \sum_{P=(X,U)\in \mP_{n,  M/n}^{(\downarrow)} }\hspace{-.2cm} \big|  \inDt(P)-\mu(U) \big|\, \one \Big\{ \big|  \inDt(P)-\mu(U) \big|\le \eta a_n  \Big\} \bigg] \le \frac{n}{\vep a_n}\, \eta a_n Mn^{-1} = \frac{\eta M}{\vep}. 
\end{align*}
Since $\eta M/\vep\to0$ as $\eta\to0$, it suffices to demonstrate that $A_n\to0$ as $n\to\infty$ for every $\eta>0$. By the union bound, 
\begin{align*}
A_n &\le \E \bigg[ \sum_{P=(X,U)\in \mP_{n,  M/n}^{(\downarrow)} } \one \Big\{ \big| \inDt(P)-\mu(U) \big|>\eta a_n \Big\}\bigg]  \\
&= n\int_0^{Mn^{-1}} \P \Big( \big| \inDt(u)-\mu(u)  \big| >\eta a_n \Big) \dif u \\
&\le n\int_0^1 \P \Big( \big| \inDt(u)-\mu(u)  \big| >\eta a_n \Big) \dif u \to 0, \ \ \ n\to\infty, 
\end{align*}
where the last convergence was already shown by \eqref{e:A_n.conv} along with Chebyshev's inequality. 
\end{proof}

Before starting the proof, we note  that parts of the calculations are similar to those appearing in the Appendix. 

\begin{proof}[Proof of Theorem \ref{t:stable.limit} $(i)$]
By Proposition \ref{p:stable.limit.uniform} $(i)$, it suffices to verify that 
\begin{align}
&\limsup_{M\to\infty}\limsup_{n\to\infty} \frac{1}{a_n^2}\, \text{Var} \Big( \sum_{P=(X,U)\in \mP_{n,  M/n}^{(\uparrow)} } \mu(U)\Big)=0,  \label{e:cond1.stable.limit.lighter} \\
&\limsup_{M\to\infty}\limsup_{n\to\infty} \frac{1}{a_n^2}\, \text{Var} \Big( \sum_{P\in \mP_{n,  M/n}^{(\uparrow)}} \inDt(P)\Big)=0,  \label{e:cond2.stable.limit.lighter} \\
&\limsup_{n\to\infty}\frac{1}{a_n}\, \E \bigg[ \bigg| \sum_{P=(X,U)\in \mP_{n,  M/n}^{(\downarrow)}}  \big( \inDt(P) - \mu(U)\big)\bigg| \bigg] = 0, \ \ \text{for every } M>0. \label{e:cond3.stable.limit.lighter}
\end{align}
For the proof of \eqref{e:cond1.stable.limit.lighter}, it follows from Karamata's theorem that 
\begin{align*}
\frac{1}{a_n^2}\, \text{Var} \Big( \sum_{P\in \mP_{n,  M/n}^{(\uparrow)}} \mu(U)\Big) &= \frac{n}{a_n^2}\, \int_{Mn^{-1}}^1 \mu(u)^2 \dif u \\
&\le \frac{n}{a_n^2}\, \int_0^{M^{-1}n} \mu(v^{-1})^2 v^{-2}\dif v \\
&\sim \frac{n}{a_n^2}\cdot \frac{1}{2\ell\gamma-1}\cdot M^{-1}n \mu (M/n)^2 (M/n)^2 \\
&=\frac{M}{2\ell\gamma-1}\Big(\frac{\mu(M/n)}{a_n}  \Big)^2. 
\end{align*}
By Proposition \ref{p:exp.asym}, we have 
$$
\mu (M/n)\sim C (M/n)^{-\ell\gamma} (\log n -\log M)^m, \ \ \ n\to\infty. 
$$
Moreover, $a_n \sim Cn^{\ell\gamma}(\log n)^{m}$. Therefore, 
$$
\limsup_{M\to\infty}\limsup_{n\to\infty}\frac{M}{2\ell\gamma-1}\Big(\frac{\mu(M/n)}{a_n}  \Big)^2 = \limsup_{M\to\infty}\limsup_{n\to\infty} \frac{M^{1-2\ell\gamma}}{2\ell\gamma-1}\, \Big( 1-\frac{\log M}{\log n} \Big)^{2m}=0. 
$$

We next deal with \eqref{e:cond3.stable.limit.lighter}. The left hand side can be bounded by 
$$
\frac{n}{a_n}\,\int_0^{Mn^{-1}} \E \big[ \big| \inDt(u) - \mu(u) \big| \big]\dif u \le \frac{n}{a_n}\,\int_0^{Mn^{-1}}  \sqrt{\text{Var}(\inDt(u))}\dif u. 
$$
First, if $m\ge1$, it follows from Proposition \ref{p:var.asym.specific} that 
\begin{align*}
\frac{n}{a_n}\,\int_0^{Mn^{-1}}  \sqrt{\text{Var}(\inDt(u))}\dif u &\le \frac{Cn}{a_n}\int_0^{Mn^{-1}} u^{-\ell\gamma} (\log u^{-1})^{m-1/2} \dif u \\
&= \frac{Cn}{a_n}\int_{M^{-1}n}^\infty v^{\ell\gamma-2} (\log v)^{m-1/2} \dif v. 
\end{align*}
Since $\gamma<1/\ell$ and $m\ge1$, the last term is asymptotically equal to 
\begin{align*}
\frac{Cn}{a_n}\cdot \frac{(M^{-1}n)^{\ell\gamma-1}}{1-\ell\gamma}\, \big(\log (M^{-1}n)\big)^{m-1/2} \sim CM^{1-\ell\gamma} \frac{(\log n -\log M)^{m-1/2}}{(\log n)^m} \to 0, \ \ \ n\to\infty. 
\end{align*}
If $m=0$, Proposition \ref{p:var.asym.specific} gives 
$$
\frac{n}{a_n}\,\int_0^{Mn^{-1}}  \sqrt{\text{Var}(\inDt(u))}\dif u \le \frac{Cn}{a_n}\int_0^{Mn^{-1}} u^{-\ell\gamma+\gamma/2}\dif u = C M^{1-\ell\gamma+\gamma/2}n^{-\gamma/2} \to 0, \ \ \ n\to\infty. 
$$

Finally, we turn our attention to \eqref{e:cond2.stable.limit.lighter}. Appealing to the explicit formula on page 217 in \cite{barysh}, which is based on the Mecke formula, for the variance of a sum of score functions, one can see that 
\begin{align}
\begin{split}  \label{e:3rd.cond.cov}
&\frac{1}{a_n^2}\, \text{Var} \Big( \sum_{P\in \mP_{n,  M/n}^{(\uparrow)}} \inDt(P)\Big) =\frac{n}{a_n^2}\int_{Mn^{-1}}^1 \E\big[\inDt(u)^2\big]\dif u \\
&\quad + \frac{n}{a_n^2} \int_0^n \int_{Mn^{-1}}^1 \int_{Mn^{-1}}^1 \text{Cov} \big( \inDt (0,u), \inDt (y,v)\big) \dif v \dif u \dif y + R_n, 
\end{split}
\end{align}
where $R_n$ is a remainder term that tends to $0$ as $n\to\infty$ more rapidly than the other terms. Since the required negligibility of $R_n$ can be demonstrated similarly to the proof in \cite[Lemma 10]{hj23}, we skip the technical details here.

We work with the first term in \eqref{e:3rd.cond.cov}. If $m\ge1$,  Propositions \ref{p:exp.asym} and \ref{p:var.asym.specific}  yield that 
\begin{align*}
\E\big[\inDt(u)^2\big] &= \text{Var}\big( \inDt(u) \big) + \big(\E[\inDt(u)]\big)^2  \\
&\le C \Big\{ u^{-2\ell\gamma} (\log u^{-1})^{2m-1} + \Big( u^{-\ell\gamma}(\log u^{-1})^m \Big)^2 \Big\} \\
&\le C u^{-2\ell\gamma} (\log u^{-1})^{2m}. 
\end{align*}
Since $2\ell\gamma-2>-1$, one can see that 
\begin{align*}
\frac{n}{a_n^2}\int_{Mn^{-1}}^1 \E\big[\inDt(u)^2\big]\dif u &\le \frac{Cn}{a_n^2}\int_0^{M^{-1}n} v^{2\ell\gamma-2} (\log v)^{2m} \dif v \\
&\sim \frac{Cn}{a_n^2} \cdot \frac{(M^{-1}n)^{2\ell\gamma-1}}{2\ell\gamma-1}\, (\log (M^{-1}n))^{2m} \\
&\le CM^{1-2\ell\gamma} \Big( \frac{\log n -\log M}{\log n} \Big)^{2m} \\
&\to CM^{1-2\ell\gamma}, \ \ \text{as } n\to\infty. 
\end{align*}
The last term converges to $0$ as $M\to\infty$. 
If $m=0$, we obtain from Propositions \ref{p:exp.asym} and \ref{p:var.asym.specific} that 
$$
\E\big[\inDt(u)^2\big] = \text{Var}\big( \inDt(u) \big) + \big(\E[\inDt(u)]\big)^2 \le Cu^{-2\ell\gamma}. 
$$
Substituting this bound, it is easy to see that 
$$
\frac{n}{a_n^2}\int_{Mn^{-1}}^1 \E\big[\inDt(u)^2\big]\dif u  \le CM^{1-2\ell\gamma} \to 0, \ \ \text{as } M\to\infty. 
$$

Regarding the second term in \eqref{e:3rd.cond.cov}, we express $\inDt(\cdot)$ as in \eqref{e:expression.Din} of the Appendix and use the Mecke formula,  to derive that 
\begin{align}
\begin{split}  \label{e:Mecke.to.covariance}
&\int_0^n \int_{Mn^{-1}}^1 \int_{Mn^{-1}}^1 \text{Cov} \big( \inDt (0,u), \inDt (y,v)\big) \dif u \dif v \dif y  \\
&= \int_0^n \int_{Mn^{-1}}^1 \int_{Mn^{-1}}^1 \E \bigg[ \sum_{(P_1, \dots, P_a)\in \mP_{\neq}^a}\sum_{\substack{(Q_1,\dots,Q_a)\in \mP_{\neq}^a, \\ |(P_1,\dots,P_a)\cap (Q_1,\dots,Q_a)|\ge 1}} \hspace{-10pt}h\big( (0,u),P_1,\dots,P_a \big)\, \\
&\qquad \qquad\qquad\qquad\qquad\qquad\qquad\qquad\times h\big( (y,v),Q_1,\dots,Q_a \big)  \bigg]  \dif u \dif v \dif y, 
\end{split}
\end{align}
where  $h$ is an indicator function defined at \eqref{e:expression.Din} and $\mP_{\neq}^a$ is also defined there. For the same reasoning as in the proof of Proposition \ref{p:var.asym.specific} in the Appendix, we discuss only the case $\big|(P_1,\dots,P_a)\cap (Q_1,\dots,Q_a)\big|=1$. 

Now, we recall that the directed tree $\ms{T}$ determined by $h$ contains $\ell$ leaves with  no further children. In the following, we separate our argument into two cases, $(i)$ $m\ge1$ and $(ii)$ $m=0$. \\
\medskip

\noindent {\bfseries Case $(i)$:} $m\ge1$. \\
Let $\ms{T}_1$ and $\ms{T}_2$ be two copies of $\ms{T}$ such that $\ms{T}_i \cong \ms{T}$ for $i=1,2$, meaning that each $\ms{T}_i$ is graph isomorphic to $\ms{T}$. By selecting a point from each of the $\ms{T}_i$ and identifying them together, we find that the product of two indicators in \eqref{e:Mecke.to.covariance} implies the presence of the following configuration of nodes and edges:
\begin{align}
\begin{split}  \label{e:cov.tree.structure}
&(p_1,q_1) \to (p_2,q_2) \to \cdots \to (p_s, q_s) \to (z_0,w_0), \\
&(p_1',q_1') \to (p_2',q_2') \to \cdots \to (p_t', q_t') \to (z_0,w_0), \\
&(z_0,w_0) \to (z_1,w_1) \to \cdots \to (z_d,w_d) \to (0,u), \\
&(z_0,w_0) \to (z_1',w_1') \to \cdots \to (z_e',w_e') \to (y,v),
\end{split}
\end{align}
for some $s, t, d, e \ge 0$. In the special case  $s=0$, the sequence of edges in the first line of \eqref{e:cov.tree.structure} is absent; similarly, if $t=0$, the sequence in the second line is absent. Likewise, if $d=0$ (resp.~$e=0$), then $(z_0, w_0)$ is directly adjacent to $(0,u)$ (resp.~$(y,v)$) without any intermediate vertices. In \eqref{e:cov.tree.structure}, $(z_0,w_0)$ corresponds to a common node between $\ms{T}_1$ and $\ms{T}_2$. Below, we focus on the case where at least one of the chosen vertices is selected from a path containing at least one intermediate node between the leaf and the root. Since $m \ge 1$, it is always possible to select two nodes in this manner. Note that the remaining case, where both chosen vertices are taken from the shortest possible path consisting only of the leaf and the root, will be addressed in Case $(ii)$. 

To further clarify the setup in \eqref{e:cov.tree.structure}, we assume, without loss of generality, that the path consisting of the first and third lines of \eqref{e:cov.tree.structure} is taken from $\ms{T}_1$, while the other path, consisting of the second and fourth lines, is taken from $\ms{T}_2$. 

Our discussion is now divided into three cases: $(i\text{-}1)$ $s \wedge t \ge 1$, $(i\text{-}2)$ $s \wedge t = 0$ and $s \vee t \ge 1$,  $(i\text{-}3)$ $s=t=0$. Below, we will address only Cases $(i\text{-}1)$ and $(i\text{-}3)$, and omit the intermediate Case $(i\text{-}2)$. We note that in Case $(i\text{-}3)$, either $d$ or $e$ must be positive, whereas both $d$ and $e$ could be zero in Cases $(i\text{-}1)$ and $(i\text{-}2)$.
\vspace{5pt}

\noindent {\bfseries Case  $(i\text{-}1)$:} $s\wedge t \ge1$. \\
We begin by offering an elementary  bound:
\begin{equation}  \label{e:elementary.bdd}
\int_{a_1}^{a_2} x^{a_3} (\log x^{-1})^{a_4} \dif x \le \begin{cases}
C a_1^{1+a_3} (\log a_1^{-1})^{a_4} & \text{ if } a_3 < -1, \\[5pt]
C(\log a_1^{-1})^{a_4 +1} & \text{ if } a_3=-1. 
\end{cases}
\end{equation}
Repeating calculations identical to those in the proof of Proposition \ref{p:exp.asym} in the Appendix, particularly by integrating over all coordinates except those in the third and fourth lines in  \eqref{e:cov.tree.structure}, it is sufficient to provide an upper bound for the following integral. 
\begin{align}
\begin{split}  \label{e:An.heavy}
A_n &:= \int_{u=Mn^{-1}}^1 \int_{v=Mn^{-1}}^1 \int_{w_d=u}^1 \int_{w_{d-1}=w_d}^1 \cdots \int_{w_1=w_2}^1 \int_{w_e'=v}^1 \int_{w_{e-1}'=w_e'}^1 \cdots \int_{w_1'=w_2'}^1 \int_{w_0=w_1\vee w_1'}^1 \\
&\qquad u^{-(\ell-1)\gamma} (\log u^{-1})^{m-(s+d)} v^{-(\ell-1)\gamma} (\log v^{-1})^{m-(t+e)} w_0^{-2\gamma} (\log w_0^{-1})^{s+t-2} \\
&\quad \times \int_{y=0}^n \prod_{i=1}^d \int_{z_i\in \R} \prod_{j=1}^e \int_{z_j'\in \R} \int_{z_0\in \R} \one \big\{ (z_0,w_0) \to (z_1,w_1) \to \cdots \to (z_d,w_d) \to (0,u) \big\} \\
&\qquad \qquad\qquad\qquad\qquad\qquad\qquad \quad  \times \one \big\{ (z_0,w_0) \to (z_1',w_1') \to \cdots \to (z_e',w_e') \to (y,v) \big\}. 
\end{split}
\end{align}
By extending the integral over $y \in [0, n]$ to $y \in \R$ and subsequently integrating over all corresponding coordinates, we observe that the expression from the third to the fourth line above satisfies
\begin{align*}
&\int_{y=0}^n \prod_{i=1}^d \int_{z_i\in \R} \prod_{j=1}^e \int_{z_j'\in \R} \int_{z_0\in \R} \one \big\{ (z_0,w_0) \to (z_1,w_1) \to \cdots \to (z_d,w_d) \to (0,u) \big\} \\
&\qquad \qquad\qquad\qquad\qquad\qquad\qquad \quad  \times \one \big\{ (z_0,w_0) \to (z_1',w_1') \to \cdots \to (z_e',w_e') \to (y,v) \big\} \\
&\le  C u^{-\gamma} v^{-\gamma}\Big( \prod_{i=1}^d w_i^{-1}\Big) \Big\{ \prod_{j=1}^e (w_j')^{-1}\Big\} w_0^{2\gamma-2}. 
\end{align*}
Referring this bound back to $A_n$, 
\begin{align*}
A_n &\le C \int_{u=Mn^{-1}}^1 u^{-\ell\gamma} (\log u^{-1})^{m-(s+d)} \int_{v=Mn^{-1}}^1 v^{-\ell\gamma} (\log v^{-1})^{m-(t+e)} \\
&\quad \times  \int_{w_d=u}^1 w_d^{-1} \int_{w_{d-1}=w_d}^1 w_{d-1}^{-1} \cdots \int_{w_1=w_2}^1 w_1^{-1}  \\
&\quad \times \int_{w_e'=v}^1 (w_e')^{-1} \int_{w_{e-1}'=w_e'}^1 (w_{e-1}')^{-1} \cdots \int_{w_1'=w_2'}^1 (w_1')^{-1} \int_{w_0=w_1\vee w_1'}^1 w_0^{-2} (\log w_0^{-1})^{s+t-2}.  
\end{align*}
By \eqref{e:elementary.bdd}, the innermost  integral can be bounded as follows: 
$$
\int_{w_0=w_1\vee w_1'}^1 w_0^{-2} (\log w_0^{-1})^{s+t-2} \le \int_{w_0=u\vee v}^1 w_0^{-2} (\log w_0^{-1})^{s+t-2} \le C(u\vee v)^{-1} \big( \log (u\vee v)^{-1} \big)^{s+t-2}. 
$$
Substituting this bound and appealing to the proof of Lemma \ref{l:sequential.integral} in the Appendix repeatedly, 
$$
A_n \le C \int_{u=Mn^{-1}}^1 \int_{v=Mn^{-1}}^1 u^{-\ell\gamma}v^{-\ell\gamma} (u \vee v)^{-1} (\log u^{-1})^{m-s} (\log v^{-1})^{m-t} \big( \log (u\vee v)^{-1} \big)^{s+t-2}. 
$$
If $v\ge u$, it follows from \eqref{e:elementary.bdd} with $\gamma>1/(2\ell)$ that 
\begin{align*}
&\int_{u=Mn^{-1}}^1 \int_{v=u}^1 u^{-\ell\gamma} v^{-1-\ell\gamma} (\log u^{-1})^{m-s} (\log v^{-1})^{m+s-2} \\
&\le C \int_{u=Mn^{-1}}^1 u^{-2\ell\gamma} (\log u^{-1})^{2m-2} \le C M^{1-2\ell\gamma} n^{2\ell\gamma-1} (\log n -\log M)^{2m-2}. 
\end{align*}
One can easily obtain the same bound even when $u\ge v$, and therefore, 
$$
A_n \le C M^{1-2\ell\gamma} n^{2\ell\gamma-1} (\log n -\log M)^{2m-2}. 
$$
We know that as $n\to\infty$, 
\begin{equation}  \label{e:n/a_n^2}
\frac{n}{a_n^2} \sim \frac{Cn}{\big( n^{-\ell\gamma} (\log n)^m \big)^2}, 
\end{equation}
and hence, $nA_n/a_n^2 \le CM^{1-2\ell\gamma}/(\log n)^2 \to 0$ as $n\to\infty$. 
\vspace{5pt}

\noindent {\bfseries Case $(i\text{-}3)$:} $s=t=0$. \\
Analogous to \eqref{e:An.heavy}, we have to upper bound the following integral. 
\begin{align}
\begin{split}  \label{e:Bn.heavy}
B_n &:= \int_{u=Mn^{-1}}^1 \int_{v=Mn^{-1}}^1 \int_{w_d=u}^1 \int_{w_{d-1}=w_d}^1 \cdots \int_{w_1=w_2}^1 \int_{w_e'=v}^1 \int_{w_{e-1}'=w_e'}^1 \cdots \int_{w_1'=w_2'}^1 \int_{w_0=w_1\vee w_1'}^1 \\
&\qquad u^{-(\ell-1)\gamma} (\log u^{-1})^{m-d} v^{-(\ell-1)\gamma} (\log v^{-1})^{m-e} \\
&\quad \times \int_{y=0}^n \prod_{i=1}^d \int_{z_i\in \R} \prod_{j=1}^e \int_{z_j'\in \R} \int_{z_0\in \R} \one \big\{ (z_0,w_0) \to (z_1,w_1) \to \cdots \to (z_d,w_d) \to (0,u) \big\} \\
&\qquad \qquad\qquad\qquad\qquad\qquad\qquad \quad  \times \one \big\{ (z_0,w_0) \to (z_1',w_1') \to \cdots \to (z_e',w_e') \to (y,v) \big\} \\
&\le C \int_{u=Mn^{-1}}^1 u^{-\ell\gamma} (\log u^{-1})^{m-d} \int_{v=Mn^{-1}}^1 v^{-\ell\gamma} (\log v^{-1})^{m-e} \\
&\quad \times  \int_{w_d=u}^1 w_d^{-1} \int_{w_{d-1}=w_d}^1 w_{d-1}^{-1} \cdots \int_{w_1=w_2}^1 w_1^{-1}  \\
&\quad \times \int_{w_e'=v}^1 (w_e')^{-1} \int_{w_{e-1}'=w_e'}^1 (w_{e-1}')^{-1} \cdots \int_{w_1'=w_2'}^1 (w_1')^{-1} \int_{w_0=w_1\vee w_1'}^1 w_0^{2\gamma-2}. 
\end{split}
\end{align}
It is evident that 
\begin{equation} \label{e:inner.bound5}
\int_{w_0=w_1\vee w_1'}^1 w_0^{2\gamma-2} \le \int_{w_0=u\vee v}^1 w_0^{2\gamma-2} \le \begin{cases}
C(u\vee v )^{2\gamma-1} & \text{ if } \ell\ge2, \\
C & \text{ if } \ell=1. 
\end{cases}
\end{equation}
Suppose first $\ell\ge2$. Then, by  using \eqref{e:inner.bound5}, 
\begin{align*}
B_n &\le C\int_{u=Mn^{-1}}^1 \int_{v=Mn^{-1}}^1 u^{-\ell\gamma} v^{-\ell\gamma} (u\vee v)^{2\gamma-1} (\log u^{-1})^m (\log v^{-1})^m \\
&= 2C \int_{u=Mn^{-1}}^1 u^{-\ell\gamma}  (\log u^{-1})^m \int_{v=u}^1 v^{-1-(\ell-2)\gamma} (\log v^{-1})^m \\
&\le 2C \int_{u=Mn^{-1}}^1 u^{-(2\ell-2)\gamma} (\log u^{-1})^{2m+1} \le CM^{1-2\ell\gamma} n^{2\ell\gamma-1}. 
\end{align*}
Because of \eqref{e:n/a_n^2} and the assumption $m\ge1$, we have 
$$
\frac{nB_n}{a_n^2} \le \frac{CM^{1-2\ell\gamma}}{(\log n)^{2m}} \to 0, \ \ \text{as } n\to\infty. 
$$
Suppose next  $\ell=1$; then 
$$
B_n \le C \Big\{ \int_{Mn^{-1}}^1 u^{-\gamma} (\log u^{-1})^m \dif u \Big\}^2 \le C, 
$$
in which case, it follows that $nB_n/a_n^2 \le Cn^{1-2\ell\gamma} \to 0$ as $n\to\infty$. 
Now, our argument for Case $(i)$ has been completed and we shall move to  Case $(ii)$. 
\vspace{5pt}

\noindent {\bfseries Case $(ii)$:} $m=0$. \\
According to the indicator in \eqref{e:Mecke.to.covariance}, it is enough to consider the following configuration of nodes and edges. 
\begin{align*}
(z_0, w_0) \to (0,u), \ \ \ (z_0,w_0) \to (y,v), \\
(z_i, w_i) \to \begin{cases}
(0,u),  & \ i=1,\dots,\ell-1, \\
(y,v), & \ i=\ell, \dots, 2\ell-2. 
\end{cases}
\end{align*}
As an analogue of \eqref{e:An.heavy} and \eqref{e:Bn.heavy}, we need to upper bound the integral: 
\begin{align*}
D_n &:= \int_{u=Mn^{-1}}^1 \int_{v=Mn^{-1}}^1 \prod_{j=\ell}^{2\ell-2} \int_{w_j=v}^1 \prod_{i=1}^{\ell-1} \int_{w_i=u}^1 \int_{w_0=u\vee v}^1 \\
&\quad \times \int_{y=0}^n \prod_{i=0}^{2\ell-2} \int_{z_i\in \R} \one \big\{  (z_0,w_0) \to (0,u), \ (z_0,w_0)\to (y,v), \\
&\qquad \qquad (z_i,w_i) \to (0,u), \, i=1,\dots,\ell-1, \ (z_j, w_j) \to (y,v), \, j=\ell, \dots, 2\ell-2 \big\}  \\
&\le  C \int_{u=Mn^{-1}}^1 u^{-\ell\gamma} \int_{v=Mn^{-1}}^1 v^{-\ell\gamma} \prod_{j=\ell}^{2\ell-2} \int_{w_j=v}^1 w_j^{\gamma-1} \prod_{i=1}^{\ell-1} \int_{w_i=u}^1 w_i^{\gamma-1} \int_{w_0=u\vee v}^1 w_0^{2\gamma-2} \\
&\le C \int_{u=Mn^{-1}}^1 u^{-\ell\gamma} \int_{v=Mn^{-1}}^1 v^{-\ell\gamma} \int_{w_0=u\vee v}^1 w_0^{2\gamma-2}. 
\end{align*}
If $\ell\ge2$, we apply the first bound in \eqref{e:inner.bound5} to obtain that 
\begin{align*}
D_n &\le C\int_{u=Mn^{-1}}^1 u^{-(2\ell-2)\gamma} \log u^{-1} \\
&\le  C\int_{u=Mn^{-1}}^1 u^{-(2\ell-1)\gamma} \log u^{-1}
\le \begin{cases}
C (Mn^{-1})^{1-(2\ell-1)\gamma} & \text{ if } 1/(2\ell-1) \le \gamma < 1/\ell, \\
C & \text{ if } 1/(2\ell) < \gamma < 1/(2\ell-1). 
\end{cases}
\end{align*}
In either case, it is easy to show that $nD_n/a_n^2\to 0$ as $n\to\infty$. If $\ell=1$,  the second bound in \eqref{e:inner.bound5} is applied to conclude  that $D_n \le C$, and thus, $nD_n/a_n^2 \le Cn^{1-2\ell\gamma}\to 0$ as $n\to\infty$. 
\end{proof}
\medskip

%
%SEC CLIQUE
%
\section{Limit theorems for clique counts}
\label{sec:clique}

In this section, we aim to establish various limit theorems for the statistics associated with cliques in the $\ms{AD}(\beta,\gamma)$. For a given $m \ge 2$, we consider the number of $m$-dimensional cliques (abbreviated as $m$-cliques) whose vertex with the lowest mark lies in $\bbT_n$, $n \ge 1$. For $(x,u) \in \bbT$, let $\inCc(x,u)$ denote the number of $m$-cliques in which $(x,u)$ is the vertex with the lowest mark. More precisely, analogous to the definition of $\inDt(x,u)$ in Section \ref{sec:tree}, $\inCc(x,u)$ represents the number of $(m-1)$-tuples $(P_1, \dots, P_{m-1}) \in \mP_{\neq}^{m-1}$ that form an $(m-1)$-clique, with the additional property that $P_i \to (x,u)$ for all $i = 1, \dots, m-1$. In other words, $(P_1, \dots, P_{m-1})$ forms an $m$-clique, together with the lowest mark vertex $(x,u)$.  If $x=0$, we write $\inCc(u) = \inCc(0,u)$. 

The count of $m$-cliques of interest is then defined by
\begin{equation}  \label{e:def.clique.counts}
\sum_{P \in \mP _n} \inCc (P), \ \ \ n \ge 1,  
\end{equation}
and the associated point process is given by
\begin{equation}  \label{e:def.clique.pp}
\sum_{P \in \mP _n}  \delta_{b_n^{-1} \inCc (P)}, \ \ \ n \ge 1, 
\end{equation}
where $(b_n)_{n \ge 1}$ are scaling constants that will be defined later. 
Analogous to the previous section, the nature of the limit theorems for the clique counts \eqref{e:def.clique.counts} is critically influenced by the value of $\gamma$. Specifically, stable limit theorems are observed only when $1/2 < \gamma < 1$, whereas \eqref{e:def.clique.counts} follows a CLT when $0 < \gamma < 1/2$, as will be shown in the companion work \cite{aa}. 

Below, Proposition \ref{p:exp.var.asym.clique} provides the asymptotics of the expectation 
$$
\nu(u) := \E\big[\inCc(x,u)\big], \ \ \ (x,u) \in \bbT,
$$
as well as an upper bound for the variance of $\inCc(x,u)$. The proof of Proposition \ref{p:exp.var.asym.clique} is presented in the Appendix. Notably, a $2$-clique corresponds to an edge, which can be viewed as  a special case of a tree. Since the asymptotics for sub-tree counts  were already examined in the previous section, we will focus exclusively on the case $m \ge 3$. 

\begin{proposition}  \label{p:exp.var.asym.clique}
$(i)$ It holds that as $u\downarrow0$, 
$$
\nu(u) \sim C_{\beta, \gamma} u^{-\gamma}, 
$$
where 
\begin{align*}
C_{\beta,\gamma} &:= \frac{2\beta}{\gamma}\prod_{i=1}^{m-2} \int_{w_i=0}^1 \prod_{j=1}^{m-2} \int_{z_j\in \R} \one \big\{ |z_k| \le \beta w_k^{-\gamma}, \, k=1,\dots,m-2\big\} \\
&\qquad \qquad \qquad \times \one \big\{ |z_k-z_\ell| \le \beta (w_k \wedge w_\ell)^{-\gamma} (w_k \vee w_\ell)^{\gamma-1}, \, k, \ell \in \{ 1,\dots,m-2 \}  \big\}
\end{align*}
$(ii)$ There exists a finite constant $C>0$, such that 
$$
\text{Var}\big( \inCc(u) \big) \le C u^{1-3\gamma} (\log u^{-1})^{m-3}, \ \ \ u\in (0,1). 
$$
\end{proposition}

\begin{remark}
By the Mecke formula for Poisson point processes, it can be readily verified that
$$
C_{\beta,\gamma} = \frac{2\beta}{\gamma}\, \E \bigg[ \sum_{(P_1,\dots, P_{m-2})\in \mP_{\neq}^{m-2}} \one \big\{ \{ (0,1), P_1,\dots, P_{m-2} \} \text{ forms an } (m-1)\text{-clique}  \big\} \bigg]. 
$$
In other words, $C_{\beta,\gamma}$ denotes the expected count of $(m-1)$-cliques that include a fixed node $(0,1)$ of the highest mark. 
\end{remark}

Before proceeding, we formally define the scaling constants in \eqref{e:def.clique.pp} as
$$
b_n := \nu\Big(\frac{1}{n}\Big) \sim C_{\beta,\gamma} n^{\gamma}, \ \ \ n\to\infty.
$$
Now, we can present a series of limit theorems for \eqref{e:def.clique.counts} and \eqref{e:def.clique.pp}.

\begin{theorem}[Point process convergence for clique counts]  \label{t:pp.weak.conv.clique}
Let $\gamma \in (1/2,1)$. Then, the point process \eqref{e:def.clique.pp} weakly converges to a Poisson point process in $(0,\infty]$ with intensity measure $\kappa_\gamma$ such that $\kappa_\gamma \big( (y,\infty] \big)=y^{-1/\gamma}$, $y>0$. 
\end{theorem}

\begin{proof}
As in the proof of Theorem \ref{t:pp.weak.conv}, we begin by demonstrating that the process $\sum_{P=(X,U)\in \mP _n}\delta_{b_n^{-1}\nu(U)}$ converges weakly to $\ms{PPP}(\kappa_\gamma)$. To this end, we claim that 
\begin{align}
&\sum_{P=(X,U)\in \mP_n}\delta_{b_n^{-1}\nu(U)} - \sum_{i=1}^n \delta_{b_n^{-1}\nu(U_i)} \stackrel{p}{\to} \emptyset \ \ \text{in } M_p\big((0,\infty]\big),  \label{e:first.pp.clique}\\
&\sum_{i=1}^n \delta_{b_n^{-1}\nu(U_i)} \Rightarrow \ms{PPP}(\kappa_\gamma), \ \ \text{in }  M_p\big((0,\infty]\big),  \label{e:second.pp.clique} 
\end{align}
where $U_1, \dots, U_n$ are i.i.d.~uniform random variables on $(0,1)$. One can show  \eqref{e:first.pp.clique} by proceeding similarly to \eqref{e:de-Poi.approx}, while \eqref{e:second.pp.clique} follows straightforwardly from 
$$
n\P \big( b_n^{-1}\nu(U_1)>y \big) \to \kappa_\gamma\big( (y,\infty] \big) = y^{-1/\gamma}, \ \ \ n\to\infty. 
$$
It now suffices to verify that 
$$
\sum_{P\in \mathcal P_n}\delta_{b_n^{-1}\inCc(P)}  - \sum_{P=(X,U)\in \mP_n}\delta_{b_n^{-1}\nu(U)}  \stackrel{p}{\to} \emptyset \ \ \text{in } M_p\big((0,\infty]\big). 
$$
For the proof, as an analogue of \eqref{e:A_n.conv}, it is sufficient to demonstrate that for every $\vep_1>0$, 
\begin{equation}  \label{e:goal.pp.clique}
\frac{n}{b_n^2}\, \int_{\vep_1/n}^1 \text{Var}\big(\inCc(u)  \big)\dif u \to 0, \ \ \text{as }  n\to\infty. 
\end{equation}
By Proposition \ref{p:exp.var.asym.clique}, it holds that $n/b_n^2\le C n^{1-2\gamma}$, and further, 
\begin{equation}  \label{e:var.clique.An}
 \int_{\vep_1/n}^1 \text{Var}\big(\inCc(u)  \big)\dif u \le C \int_{\vep_1/n}^1 u^{1-3\gamma-\eta}\dif u, 
\end{equation}
where $\eta\in (0,1-\gamma)$. If $\frac{1}{2}<\gamma \le \frac{2-\eta}{3}$, then \eqref{e:var.clique.An} is at most finite (up to logarithmic factors), and thus, the left hand side of \eqref{e:goal.pp.clique} is bounded by $Cn^{1-2\gamma}$ (up to logarithmic factors), which clearly tends to $0$ as $n\to\infty$. If $\frac{2-\eta}{3}<\gamma<1$, then 
$$
\frac{n}{b_n^2}\, \int_{\vep_1/n}^1 \text{Var}\big(\inCc(u)  \big)\dif u  \le Cn^{1-2\gamma} \Big( \frac{\vep_1}{n} \Big)^{2-3\gamma-\eta} =Cn^{\gamma-1+\eta} \to 0, \ \ \ n\to\infty. 
$$
\end{proof}

\begin{corollary}  \label{cor:max.clique}
Let $\gamma\in (1/2,1)$. We have, as $n\to\infty$, 
$$
\frac{1}{b_n}\bigvee_{P\in \mP_n} \inCc(P) \Rightarrow Y_{\gamma}, 
$$
where  $Y_{\gamma}$ is a $1/\gamma$-Fr\'echet random variable. 
\end{corollary}

The proof of Corollary \ref{cor:max.clique} is completely identical to that of Corollary \ref{cor:maxima}, so we will omit it.

Finally, we present a stable limit theorem for the clique counts \eqref{e:def.clique.counts}. Unlike Theorem \ref{t:stable.limit}, there is no phase transition with respect to the values of $\gamma$.

\begin{theorem}[Stable limit theorem for clique counts]
Let $\gamma \in (1/2,1)$. Then,  as $n\to\infty$, 
$$
\frac{1}{b_n}\, \Big( \sum_{P\in \mP_n} \inCc(P) - \E \Big[  \sum_{P\in \mP_n} \inCc(P)  \Big] \Big) \Rightarrow S_{1/\gamma}, 
$$
where $S_{1/\gamma}$ is a zero-mean $1/\gamma$-stable random variable, defined by \eqref{e:ch.f.stable} with $\ell=1$. 
\end{theorem}

\begin{remark}
As explained in Section \ref{sec:tree}, here we do not consider the degenerate case  $\ga \ge 1$. Indeed, in that parameter regime, each node  would have infinite degree with probability 1.
%{\color{red} A remark should be inserted to discuss why we need the assumption $\gamma<1$.}
\end{remark}

\begin{proof}
As a first step, we need to establish a stable limit theorem analogous to Proposition \ref{p:stable.limit.uniform} $(i)$. Specifically, we claim that as $n \to \infty$,
\begin{equation}  \label{e:stable.limit.uniform.clique}
\frac{1}{b_n}\, \Bigg( \sum_{P=(X,U)\in \mP _n} \nu(U) - \E \Bigg[  \sum_{P=(X,U)\in \mP _n} \nu(U) \Bigg] \Bigg) \Rightarrow S_{1/\gamma}. 
\end{equation}
The proof of \eqref{e:stable.limit.uniform.clique}, however, is essentially a repetition of that of Proposition \ref{p:stable.limit.uniform} $(i)$ with appropriate modifications, so we shall omit the detailed arguments.

Now, as an analogue of \eqref{e:cond1.stable.limit.lighter}--\eqref{e:cond3.stable.limit.lighter}, the remainder of the argument will be dedicated to demonstrating the following results.
\begin{align}
&\lim_{n\to\infty}\frac{1}{b_n^2}\, \text{Var} \Big( \sum_{P=(X,U)\in \mP_{n,  n^{-1 + \xi}}^{(\uparrow)}} \nu(U)\Big)=0,  \label{e:cond1.stable.limit.clique} \\
&\lim_{n\to\infty}\frac{1}{b_n^2}\, \text{Var} \Big( \sum_{P\in \mP_{n,  n^{-1 + \xi}}^{(\uparrow)}} \inCc(P)\Big)=0,  \label{e:cond2.stable.limit.clique} \\
&\lim_{n\to\infty}\frac{1}{b_n}\, \E \bigg[ \sum_{P=(X,U)\in \mP_{n,  n^{-1 + \xi}}^{(\downarrow)}} \big|\,  \inCc(P) - \nu(U) \big| \bigg] = 0, \label{e:cond3.stable.limit.clique}
\end{align}
where $\xi$ is a finite and positive constant that will be determined shortly. 

For the proof of  \eqref{e:cond3.stable.limit.clique}, fix $\eta \in \big(0,(1-\gamma)/2\big)$ and choose 
$$
\xi \in \Big( 0,\, \frac{1-\gamma-\eta}{3(1-\gamma)-\eta} \Big). 
$$
By the Cauchy-Schwarz inequality and Proposition \ref{p:exp.var.asym.clique}, 
\begin{align*}
&\frac{1}{b_n}\, \E \bigg[ \sum_{P=(X,U)\in \mP_{n,  n^{-1 + \xi}}^{(\downarrow)}} \big|\, \inCc(P) - \nu(U) \big| \bigg]  \\
&= \frac{n}{b_n}\, \int_0^{n^{-1+\xi}} \E \Big[ \big|\inCc(u) -\nu(u) \big| \Big]\dif u \le  \frac{n}{b_n}\, \int_0^{n^{-1+\xi}} \sqrt{\text{Var}\big( \inCc(u) \big)} \dif u\\
&\le Cn^{1-\gamma}  \int_0^{n^{-1+\xi}} u^{\frac{1}{2}-\frac{3\gamma}{2}-\frac{\eta}{2}} \dif u = C n^{-\frac{1}{2}(1-\gamma)+\frac{\eta}{2}+\frac{\xi}{2}(3-3\gamma-\eta)} \to 0, \ \ \ n\to\infty. 
\end{align*}
The last convergence is assured by the constraint on $\xi$. 

To show \eqref{e:cond1.stable.limit.clique}, we apply Proposition \ref{p:exp.var.asym.clique} to obtain that,  as $n\to\infty$, 
\begin{align*}
\frac{1}{b_n^2}\, \text{Var} \Big( \sum_{P=(X,U)\in \mP_{n,  n^{-1 + \xi}}^{(\uparrow)}} \nu(U)\Big) &=\frac{n}{b_n^2}\, \int_{n^{-1+\xi}}^1 \nu(u)^2 \dif u \\
&\le Cn^{1-2\gamma} \int_{n^{-1+\xi}}^1  u^{-2\gamma} \dif u \le Cn^{\xi(1-2\gamma)} \to 0. 
\end{align*}

Finally, for the proof of \eqref{e:cond2.stable.limit.clique}, the Mecke formula yields that 
\begin{align*}
&\frac{1}{b_n^2}\, \text{Var} \Big( \sum_{P\in \mP_{n,  n^{-1 + \xi}}^{(\uparrow)}} \inCc(P)\Big) = \frac{n}{b_n^2}\, \int_{n^{-1+\xi}}^1 \E \big[ \inCc(u)^2 \big]\dif u \\
&\qquad +\frac{n}{b_n^2}\, \int_0^n \int_{n^{-1+\xi}}^1 \int_{n^{-1+\xi}}^1 \text{Cov}\big( \inCc(0,u), \, \inCc(y,v) \big) \dif u \dif v \dif y + R_n. 
\end{align*}
Here, as in \eqref{e:3rd.cond.cov}, $R_n$ represents a negligible term that consists of the sum of the remaining terms. 
%denotes the sum of the remaining terms. The first term on the right-hand side is bounded by $Cn^{\xi(1-2\gamma)} \to 0$ as $n\to\infty$. The second term is also negligible, as can be shown by following the same arguments as in the proof of Proposition \ref{p:exp.var.asym.clique} $(ii)$.
%{\color{red} \text{extra negligible terms}}
Proposition \ref{p:exp.var.asym.clique} ensures that 
\begin{align*}
\E \big[ \inCc(u)^2 \big] =\text{Var}\big( \inCc(u) \big) + \nu(u)^2 \le C\big( u^{1-3\gamma}(\log u^{-1})^{m-3} + u^{-2\gamma} \big) \le Cu^{-2\gamma}, 
\end{align*}
and hence, as $n\to\infty$, 
$$
\frac{n}{b_n^2}\, \int_{n^{-1+\xi}}^1 \E \big[ \inCc(u)^2 \big]\dif u \le Cn^{\xi(1-2\gamma)} \to 0. 
$$
Using an indicator function $g$ at \eqref{e:def.g} in the Appendix, together with the Mecke formula, 
\begin{align*}
&\int_0^n \int_{n^{-1+\xi}}^1 \int_{n^{-1+\xi}}^1 \text{Cov}\big( \inCc(0,u), \, \inCc(y,v) \big) \dif u \dif v\dif y \\
&=\sum_{q=1}^{m-1} \int_0^n \int_{n^{-1+\xi}}^1 \int_{n^{-1+\xi}}^1 \\
&\quad \E \bigg[ \sum_{\substack{(P_1,\dots,P_{m-1})\in  \mP_{\neq}^{m-1}}} \hspace{-20pt}\sum_{\substack{(Q_1,\dots,Q_{m-1})\in \mP_{\neq}^{m-1}, \\ |(P_1,\dots,P_{m-1})\cap (Q_1,\dots,Q_{m-1})|=q}} \hspace{-30pt} g\big( (0,u), P_1,\dots,P_{m-1} \big)\, g\big( (y,v), Q_1,\dots,Q_{m-1} \big)  \bigg]  \dif u \dif v \dif y  \\
&=: \sum_{q=1}^{m-1} \int_0^n \int_{n^{-1+\xi}}^1 \int_{n^{-1+\xi}}^1 J_q\big( u, (y,v)\big)  \dif u \dif v\dif y. 
\end{align*}
As in the proof of Proposition \ref{p:exp.var.asym.clique} $(ii)$ in the Appendix, we consider the following configuration of nodes. First, let $(z,w)$ be the node with the highest mark among all $q$ common points. Denote by $(y_1,v_1), \dots, (y_{m-2}, v_{m-2})$, with $v_1 \ge \cdots \ge v_{m-2}$, a set of nodes in the $m$-clique that includes $(0,u)$ and $(z,w)$. Additionally, let $(z_1,w_1), \dots, (z_{m-q-1}, w_{m-q-1})$, with $w_1 \ge \cdots \ge w_{m-q-1}$, represent the remaining nodes in the other $m$-clique that contains $(y,v)$. Our discussion below is divided into two cases.
\vspace{5pt}

\noindent {\bfseries Case $(i)$:} $v_1 \wedge w_1>w$. \\
To derive a suitable upper bound for $J_q\big( u, (y,v) \big)$, we consider the following configuration of nodes and edges, similar to \eqref{e:conf.v1.w1.larger.than.w} in the Appendix: 
\begin{align}
\begin{split}  \label{e:conf.v1.w1.larger.than.w.covariance}
&(y_1,v_1) \to \cdots \to (y_{m-2},v_{m-2}) \to (0,u), \\
&(z_1,w_1) \to \cdots \to (z_{m-q-1}, w_{m-q-1}) \to (y,v), \\
&(y_1,v_1) \to (z,w), \ (z_1,w_1) \to (z,w); 
\end{split}
\end{align}
see also Figure \ref{fig:clique.configuration.cov}. 

\begin{figure}
\includegraphics[scale=0.4]{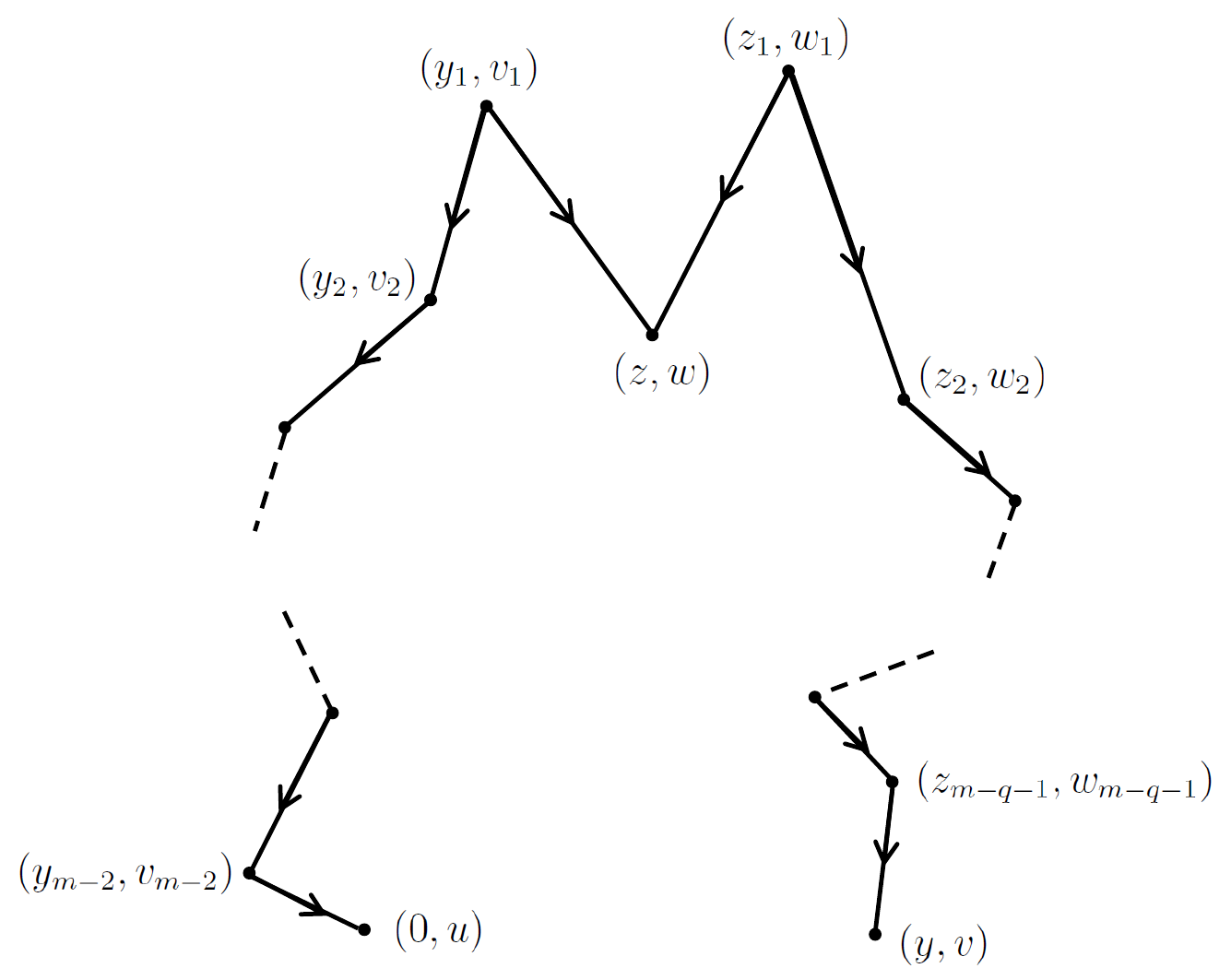}
\caption{\label{fig:clique.configuration.cov} \footnotesize{Configuration of nodes and edges for the proof of Case $(i)$.}}
\end{figure}

We first consider the case $m\ge4$. By the triangle inequality and \eqref{e:conf.v1.w1.larger.than.w.covariance}, we have that 
$$
|y| \le (2m-q-1)\beta (u\wedge v)^{-\gamma} (v_{m-2}\wedge w_{m-q-1})^{\gamma-1} \le (2m-q-1)\beta (u\wedge v)^{-1}; 
$$
hence, it follows that 
\begin{equation}  \label{e:s.wedge}
v_{m-2}\wedge w_{m-q-1} \le \Big(\frac{(2m-q-1)\beta}{(u\wedge v)^\gamma |y|}  \Big)^{1/(1-\gamma)} \one \Big\{ |y| \le \frac{(2m-q-1)\beta}{u \wedge v}\Big\} =: s\big( u\wedge v, |y| \big). 
\end{equation}
For ease of description, we set $(2m-q-1)\beta \equiv 1$ and define 
$$
s_\wedge \big( u\wedge v, |y| \big) := s\big( u\wedge v, |y| \big) \wedge 1. 
$$
Under such  configuration of nodes and edges, one needs to upper bound the integral 
\begin{align}
\begin{split}  \label{e:initial.configuration}
&\int_{v_{m-2}=u}^1 \int_{v_{m-3}=v_{m-2}}^1 \cdots \int_{v_2=v_3}^1 \int_{w_{m-q-1}=v}^1 \int_{w_{m-q-2}=w_{m-q-1}}^1 \cdots \int_{w_2=w_3}^1 \\
&\times \int_{w=u\vee v}^1 \int_{v_1=w}^1 \int_{w_1=w}^1 \prod_{i=1}^{m-2} \int_{y_i\in \R} \int_{z\in \R} \prod_{j=1}^{m-q-1} \int_{z_j\in \R} \one \big\{ (y_1,v_1) \to \cdots \to (y_{m-2},v_{m-2}) \to (0,u) \big\} \\
&\times \one \big\{ (z_1,w_1)\to \cdots \to (z_{m-q-1}, w_{m-q-1}) \to (y,v), \, (y_1,v_1) \to (z,w), \, (z_1,w_1) \to (z,w) \big\} \\
&\times \one \big\{ v_{m-2}\wedge w_{m-q-1} \le s_\wedge (u\wedge v, |y|) \big\}. 
\end{split}
\end{align}
The subsequent argument will vary depending on whether $w_{m-q-1} \le v_{m-2}$ or $v_{m-2} \le w_{m-q-1}$.
\medskip

\noindent {\bfseries Case $(i\text{-}1)$:} $w_{m-q-1} \le v_{m-2}$. \\
Dropping the condition $(z_{m-q-1}, w_{m-q-1}) \to (y,v)$ from the above indicator, while  integrating the remaining indicators over spatial coordinates, we now have to bound 
\begin{align*}
A_{u, (y,v)} &:= u^{-\gamma} \int_{v_{m-2}=u}^1 v_{m-2}^{-1} \int_{v_{m-3}=v_{m-2}}^1 v_{m-3}^{-1} \times \cdots\times  \int_{v_2=v_3}^1 v_2^{-1} \\
&\quad \times \int_{w_{m-q-1}=0}^{s_\wedge (u\wedge v, |y|)}  w_{m-q-1}^{-\gamma}\int_{w_{m-q-2}=w_{m-q-1}}^1 w_{m-q-2}^{-1}\times  \cdots\times  \int_{w_2=w_3}^1w_2^{-1} \\
&\quad \times \int_{w=u\vee v}^1 w^{-2\gamma}\int_{v_1=w}^1 v_1^{2(\gamma-1)} \int_{w_1=w}^1w_1^{2(\gamma-1)}. 
\end{align*}
For the sake of convenience in our discussion, we will disregard the appearance of logarithmic terms throughout the following argument. It is then straightforward to obtain that
\begin{equation}  \label{e:bound.A.u.y.v}
A_{u, (y,v)} \le C u^{-\gamma} (u \vee v)^{1-2\gamma} s_\wedge (u \wedge v, |y|)^{1-\gamma}.
\end{equation}
Thus, 
\begin{align*}
\int_0^n \int_{n^{-1+\xi}}^1 \int_{n^{-1+\xi}}^1  A_{u, (y,v)} \dif u \dif v \dif y \le C\int_{n^{-1+\xi}}^1 \int_{n^{-1+\xi}}^1  u^{-\gamma}  (u\vee v )^{1-2\gamma} \int_0^\infty s_\wedge (u\wedge v, b)^{1-\gamma} \dif b \dif u\dif v. 
\end{align*}
Observe now that 
\begin{equation}  \label{e:bound.s.wedge}
\int_0^\infty s_\wedge (u\wedge v, b)^{1-\gamma} \dif b \le (u\wedge v)^{-\gamma} + (u\wedge v)^{-\gamma} \int_{(u\wedge v)^{-\gamma}}^{(u\wedge v)^{-1}} b^{-1}\dif b \le C(u\wedge v)^{-\gamma}, 
\end{equation}
where at the last step, we have again ignored a logarithmic term. Appealing to this  bound, one can see that 
\begin{align}
\begin{split} \label{e:A.u.yv.into.Bn.Cn}
\int_0^n \int_{n^{-1+\xi}}^1 \int_{n^{-1+\xi}}^1  A_{u, (y,v)} \dif u \dif v \dif y  
&\le C \int_{v=n^{-1+\xi}}^1 v^{-\gamma} \int_{u=v}^1u^{1-3\gamma} + C \int_{u=n^{-1+\xi}}^1 u^{-2\gamma} \int_{v=u}^1 v^{1-2\gamma} \\
&=: B_n + C_n. 
\end{split}
\end{align}
It is easy to calculate that 
$$
B_n \le \begin{cases}
C &\text{ if } 1/2<\gamma \le 3/4, \\
Cn^{(1-\xi)(4\gamma-3)} & \text{ if } 3/4 <\gamma<1, 
\end{cases}
$$
while it holds that $C_n\le Cn^{(1-\xi)(2\gamma-1)}$ for all $1/2<\gamma<1$. Now, it follows that $nb_n^{-2}(B_n+C_n) \le Cn^{1-2\gamma} (B_n+C_n) \to 0$ as $n\to\infty$. 
\medskip

\noindent {\bfseries Case $(i\text{-}2)$:} $v_{m-2} \le w_{m-q-1}$. \\
In this case, by removing the condition $(y_{m-2}, v_{m-2}) \to (0,u)$ from the indicator in \eqref{e:initial.configuration}, we integrate the remaining indicators over spatial coordinates. Our task is then to provide an upper bound for 
\begin{align*}
B_{u, (y,v)} &:= v^{-\gamma} \int_{v_{m-2}=0}^{s_\wedge (u\wedge v, |y|)} v_{m-2}^{-\gamma} \int_{v_{m-3}=v_{m-2}}^1 v_{m-3}^{-1} \times \cdots\times  \int_{v_2=v_3}^1 v_2^{-1} \\
&\quad \times \int_{w_{m-q-1}=v}^1  w_{m-q-1}^{-1}\int_{w_{m-q-2}=w_{m-q-1}}^1 w_{m-q-2}^{-1}\times  \cdots\times  \int_{w_2=w_3}^1w_2^{-1} \\
&\quad \times \int_{w=u\vee v}^1 w^{-2\gamma}\int_{v_1=w}^1 v_1^{2(\gamma-1)} \int_{w_1=w}^1w_1^{2(\gamma-1)}. 
\end{align*}
Repeating the same argument as in Case $(i\text{-}1)$, we can conclude that 
$$
B_{u, (y,v)} \le Cv^{-\gamma} (u\vee v )^{1-2\gamma} s_\wedge (u\wedge v, |y|)^{1-\gamma}. 
$$ 
By the symmetry with \eqref{e:bound.A.u.y.v}, it immediately follows that 
$$
\frac{n}{b_n^2}\, \int_0^n \int_{n^{-1+\xi}}^1 \int_{n^{-1+\xi}}^1  B_{u, (y,v)} \dif u \dif v \dif y \to 0, \ \ \ n\to\infty. 
$$

Next, we consider the case $m=3$. Due to the restriction $v_1 \wedge w_1 > w$, it must be that $q=1$. Therefore, our configuration of nodes and edges can be described as
\begin{align*}
(y_1,v_1) \to (0,u), \ (z_1, w_1) \to (y,v), \ (y_1,v_1) \to (z,w), \ (z_1, w_1) \to (z,w).
\end{align*}
Analogous to \eqref{e:s.wedge}, we have $v_1 \wedge w_1 \le s_\wedge (u \wedge v, |y|)$ up to constant factors. Suppose first that $w_1 \le v_1$. In this case, we remove the condition $(z_1,w_1) \to (y,v)$ and aim to bound 
\begin{align*}
C_{u, (y,v)} &:= u^{-\gamma} \int_{w=u\vee v }^1 w^{-2\gamma} \int_{v_1=w}^1 v_1^{2(\gamma-1)}\int_{w_1=0}^{s_\wedge (u\wedge v, |y|)} w_1^{\gamma-1} \le Cu^{-\gamma} (u\vee v )^{1-2\gamma} s_\wedge (u\wedge v, |y|)^\gamma. 
\end{align*}
Then, 
$$
\int_0^n \int_{n^{-1+\xi}}^1 \int_{n^{-1+\xi}}^1  C_{u, (y,v)} \dif u \dif v \dif y \le C  \int_{n^{-1+\xi}}^1 \int_{n^{-1+\xi}}^1u^{-\gamma} (u\vee v )^{1-2\gamma}  \int_0^\infty s_\wedge (u\wedge v, b)^\gamma \dif b\dif u \dif v. 
$$
By the argument similar to \eqref{e:bound.s.wedge}, we have that 
\begin{equation}  \label{e:s.wedge.gamma}
\int_0^\infty s_\wedge (u\wedge v, b)^\gamma \dif b \le C(u\wedge v)^{-\gamma}. 
\end{equation}
Now, by substituting \eqref{e:s.wedge.gamma}, one can get the same bound as in \eqref{e:A.u.yv.into.Bn.Cn}; thus, 
$$
\frac{n}{b_n^2}\, \int_0^n \int_{n^{-1+\xi}}^1 \int_{n^{-1+\xi}}^1  C_{u, (y,v)} \dif u \dif v \dif y \to 0, \ \ \ n\to\infty. 
$$
Suppose next that  $v_1<w_1$, then we must drop the condition $(y_1,v_1)\to (0,u)$ and have to bound 
$$
D_{u, (y,v)} := v^{-\gamma} \int_{w=u\vee v }^1 w^{-2\gamma} \int_{v_1=0}^{s_\wedge (u\wedge v, |y|)} v_1^{\gamma-1} \int_{w_1=w}^1 w_1^{2(\gamma-1)}  \le Cv^{-\gamma} (u\vee v )^{1-2\gamma} s_\wedge (u\wedge v, |y|)^\gamma. 
$$
The rest of the discussion is almost identical to the last case. 
\vspace{5pt}

\noindent {\bfseries Case $(ii)$:} $v_1 \wedge w_1 \le w \le  v_1 \vee w_1$ or $w \ge v_1 \vee w_1$. \\
As an analogue of \eqref{e:clique.counts.bdd} in the Appendix,  the corresponding terms in $J_q\big( u, (y,v) \big)$ can be bounded by 
\begin{align}
\begin{split}  \label{e:top.Poisson.pt}
&\E \Big[ \sum_{P\in \mP \cap (N^\uparrow (0,u)\cap N^\uparrow(y,v))} \mP \big( N^\downarrow (P) \big)^{2m-q-3}\Big] \\
&= \int_0^1 \int_{\R} \one \big\{ (z,w)\to (0,u), \, (z,w)\to (y,v) \big\} \E \Big[ \mP\big(  N^\downarrow (z,w)\big)^{2m-q-3} \Big] \dif z \dif  w \\
&\le C \int_0^1 \int_{\R} \one \big\{ |z|\le \beta u^{-\gamma} w^{\gamma-1}, \, |z-y|\le \beta v^{-\gamma} w^{\gamma-1} \big\} \dif z \dif w. 
\end{split}
\end{align}
Since $|y|\le |y-z|+|z| \le 2\beta (u\wedge v)^{-\gamma}w^{\gamma-1}$, it follows that $w\le s_{\wedge}(u\wedge v, |y|)$ up to constant factors. Now, the last term in \eqref{e:top.Poisson.pt} is further bounded by 
$$
C\int_0^{s_\wedge (u\wedge v, |y|)}  (u\vee v)^{-\gamma}w^{\gamma-1} \dif w = C (u\vee v)^{-\gamma} s_\wedge (u\wedge v, |y|)^\gamma. 
$$
Finally, it follows from \eqref{e:s.wedge.gamma} that 
\begin{align*}
&\frac{n}{b_n^2}\, \int_0^n \int_{n^{-1+\xi}}^1\int_{n^{-1+\xi}}^1 (u\vee v)^{-\gamma} s_\wedge (u\wedge v, |y|)^\gamma \dif u \dif v \dif y \\
&\le \frac{Cn}{b_n^2}\,  \int_{n^{-1+\xi}}^1 u^{-\gamma} \dif u \le Cn^{1-2\gamma} \to 0, \ \ \ n\to\infty, 
\end{align*}
as desired. 
\end{proof}

%
%SEC PROOFS
%
\section{Appendix: proofs of expectation and variance asymptotics with a fixed lowest mark} 
\label{sec:proofs}

\subsection{Proof of Proposition \ref{p:exp.asym}}
We begin with the following lemma.
\begin{lemma} \label{l:sequential.integral}
For every $x\in \R$ and $d\ge1$, we have as $u\to0$, 
\begin{align*}
&\int_{w_d=u}^1 \int_{w_{d-1} = w_d}^1 \cdots \int_{w_1=w_2}^1 \int_{z_d\in \R} \cdots \int_{z_1\in \R} \one \big\{ (z_1,w_1)\to \cdots \to (z_d,w_d) \to (x,u) \big\} \\
&\quad \sim \frac{(2\beta)^d}{\gamma (d-1)!}\, u^{-\gamma} (\log u^{-1})^{d-1}. 
\end{align*}
Furthermore, the integral of the left  hand side is increasing as $u\downarrow 0$. 
\end{lemma}
\begin{proof}
We first integrate an indicator function over all spatial coordinates: 
$$
 \prod_{i=1}^d \int_{z_i\in \R} \one \big\{ (z_1,w_1)\to \cdots \to (z_d,w_d) \to (x,u) \big\} = (2\beta)^d w_1^{\gamma-1} \Big( \prod_{i=2}^d w_i^{-1} \Big) u^{-\gamma}. 
$$
Next, the right hand side above is integrated with respect to $w_1$: 
\begin{align}
\begin{split}  \label{e:Au.and.Bu}
&(2\beta)^d\int_{w_d=u}^1  \int_{w_{d-1} = w_d}^1 \cdots \int_{w_1=w_2}^1 w_1^{\gamma-1} \Big( \prod_{i=2}^d w_i^{-1} \Big) u^{-\gamma} \\
&\quad = \frac{(2\beta)^d}{\gamma}\, u^{-\gamma} \int_{w_d=u}^1 w_d^{-1} \int_{w_{d-1} = w_d}^1 w_{d-1}^{-1} \times \cdots\times  \int_{w_3=w_4} w_3^{-1}\int_{w_2=w_3}^1 w_2^{-1} \\
&\qquad - \frac{(2\beta)^d}{\gamma}\, u^{-\gamma} \int_{w_d=u}^1 w_d^{-1} \int_{w_{d-1} = w_d}^1 w_{d-1}^{-1} \times \cdots\times  \int_{w_3=w_4} w_3^{-1}\int_{w_2=w_3}^1 w_2^{-1+\gamma} \\
&=: A_u  - B_u. 
\end{split}
\end{align}
Applying  the identity 
$$
\int x^{-1}(\log x )^k \dif x =\frac{(\log x )^{k+1}}{k+1}, \ \ \ k\ge1, 
$$
repeatedly, it is easy to derive that 
$$
A_u = \frac{(2\beta)^d u^{-\gamma} (\log u^{-1})^{d-1}}{\gamma (d-1)!}. 
$$

For the $B_u$ in \eqref{e:Au.and.Bu}, using an obvious bound $\int_{w_2=w_3}^1 w_2^{-1+\gamma}\le \gamma^{-1}$ and, once again, repeatedly applying the same identity as above, it holds that 
$$
B_u \le Cu^{-\gamma} (\log u^{-1})^{d-2} = o(A_u), \ \ \text{as } u\downarrow 0. 
$$
Now, the proof of Lemma \ref{l:sequential.integral} is completed. 
\end{proof}
\begin{proof}[Proof of \eqref{e:exp.asym}]
To treat a more general tree, we consider the following configuration of nodes and edges. 
\begin{align}
\begin{split}  \label{e:general.tree.exp.asym}
&(y_{i,1}, v_{i,1}) \to (y_{i,2}, v_{i,2}) \to \cdots \to (y_{i, s_i+1}, v_{i, s_i+1}) \to (z_1,w_1), \ \ i=1,\dots,p, \\
&(z_1,w_1) \to (z_2,w_2) \to \cdots \to (z_d,w_d) \to (x,u). 
\end{split}
\end{align}
Let $t=\sum_{i=1}^p s_i$. According to Lemma \ref{l:sequential.integral}, the contribution of the structure in \eqref{e:general.tree.exp.asym} is, up to constant factors and asymptotically, given by 
\begin{align}
\begin{split}  \label{e:one.step.tree}
&\int_{w_d=u}^1 \int_{w_{d-1}=w_{d}}^1 \cdots \int_{w_1=w_2}^1  w_1^{-p\gamma} (\log w_1^{-1})^t  \\
&\qquad \times \prod_{i=1}^d\int_{z_i\in \R} \big\{ (z_1,w_1) \to (z_2,w_2) \to \cdots \to (z_d,w_d) \to (x,u) \big\} \\
&= (2\beta)^d u^{-\gamma} \int_{w_d=u}^1 w_d^{-1} \int_{w_{d-1}=w_d}^1 w_{d-1}^{-1} \times \cdots \times \int_{w_2=w_3}^1 w_2^{-1} \int_{w_1=w_2}^1 w_1^{-1-(p-1)\gamma} (\log w_1^{-1})^t. 
\end{split}
\end{align}
Note that the last expression is clearly increasing as $u\downarrow 0$. By the integration by parts formula, 
\begin{equation}  \label{e:integration.by.parts}
\int_{w_1=w_2}^1 w_1^{-1-(p-1)\gamma} (\log w_1^{-1})^t = \frac{w_2^{-(p-1)\gamma}}{(p-1)\gamma} (\log w_2^{-1})^t - \frac{t}{(p-1)\gamma}\int_{w_1=w_2}^1  w_1^{-1-(p-1)\gamma} (\log w_1^{-1})^{t-1}. 
\end{equation}
In the above, the first term asymptotically dominates the second one as $u\downarrow0$. Appealing to \eqref{e:integration.by.parts} repeatedly, we conclude that \eqref{e:one.step.tree} is asymptotically equal to 
$$
\Big(\frac{2\beta}{\gamma(p-1)}\Big)^d u^{-p\gamma} (\log u^{-1})^t.   
$$
From the above calculations, we observe that when multiple paths begin at each leaf and merge together, a logarithmic factor no longer contributes until the merged path reaches the root. Specifically, the exponent of $\log u^{-1}$ is solely determined by $t=\sum_{i=1}^p s_i$ and remains unaffected by $d$. The proof can be finalized by repeating these calculations whenever multiple paths begin at each leaf and subsequently merge.
\end{proof}

\subsection{Proof of Proposition \ref{p:var.asym.specific}}
Given a directed tree $\ms{T}$ defined on $a+1$ nodes with root $\ms{r}$ as considered prior to  Proposition \ref{p:var.asym.specific}, we first express $\inDt(u)$ as follows. 
\begin{equation}  \label{e:expression.Din}
\inDt(u) = \sum_{(P_1,\dots,P_a)\in \mP_{\neq}^a} h\big( (0,u), P_1,\dots, P_a \big), 
\end{equation}
where $h$ is an appropriately defined indicator  for the existence of an injective graph homomorphism from $\ms{T}$ to an induced graph on vertices $\big\{ (0,u), P_1,\dots,P_a \big\}$ such that $\ms{r}$ is mapped to $(0,u)$. Moreover, we define 
$$
\mP_{\neq}^a := \big\{ (P_1,\dots, P_a) \in \mP^a: P_i \neq P_j \text{ for } i \neq j \big\}. 
$$

By the Mecke formula for Poisson point processes, 
\begin{align}
\begin{split}  \label{e:var.mecke}
\text{Var}\big( \inDt(u)\big) &=\E \bigg[ \sum_{(P_1, \dots, P_a)\in \mP_{\neq}^a}\sum_{(Q_1,\dots,Q_a)\in \mP_{\neq}^a} \hspace{-10pt}h\big( (0,u),P_1,\dots,P_a \big)\, \\
&\qquad \qquad\qquad\qquad\qquad\qquad\qquad\qquad\times h\big( (0,u),Q_1,\dots,Q_a \big)  \bigg] - \mu(u)^2 \\
&=\mu(u) + \E \bigg[ \sum_{(P_1, \dots, P_a)\in \mP_{\neq}^a}\sum_{\substack{(Q_1,\dots,Q_a)\in \mP_{\neq}^a, \\ 1\le |(P_1,\dots,P_a)\cap (Q_1,\dots,Q_a)|\le a-1}} \hspace{-10pt}h\big( (0,u),P_1,\dots,P_a \big)\, \\
&\qquad \qquad\qquad\qquad\qquad\qquad\qquad\qquad\times h\big( (0,u),Q_1,\dots,Q_a \big)  \bigg]. 
\end{split}
\end{align}
The asymptotics of $\mu(u)$ is  given in \eqref{e:exp.asym}, so we here work with the second term above. We note that since $\inDt(u)$ is a U-statistic, the expansion could also be deduced from \cite[Lemma 3.5]{rs13}.

Here, we consider the case where $|(P_1,\dots,P_a)\cap (Q_1, \dots, Q_a)|=1$.  If more than one point is common between $(P_1,\dots,P_a)$ and $(Q_1,\dots,Q_a)$, the case requiring special attention arises when 
the product of the two indicator functions in \eqref{e:var.mecke}
 implies the presence of cycle structures. Consider, for instance, the following configuration of nodes and edges: 
\begin{align*}
&R_1 \to R_2 \to \cdots \to R_{k-1} \to R_k, \ \ \ R_1 \to R_2' \to \cdots \to R_{\ell-1}' \to R_k,
\end{align*}
where $R_i$ and $R_i'$ are distinct points in $\mP$. In such a case, it is sufficient to ``break" the cycle by removing one of the edges formed by the pair of nodes with the highest marks, i.e., either $R_1\to R_2$ or $R_1\to R_2'$. For each occurrence of such a cycle, we systematically remove edges using this approach until at most one cycle remains. Then, we apply the argument outlined below for the case where $|(P_1,\dots,P_a)\cap (Q_1, \dots, Q_a)|=1$.

Now, assuming $\big|(P_1,\dots,P_a)\cap (Q_1,\dots,Q_a)\big|=1$, we begin our discussion with the case where $m \ge 1$. By selecting a point from each of the trees $\ms{T}_1$ and $\ms{T}_2$ such that $\ms{T}_i \cong \ms{T}$ for $i=1,2$ (i.e., $\ms{T}_i$ is graph isomorphic to $\ms{T}$), and identifying these selected points together, we can observe  the following structure:
\begin{align}  
\begin{split}  \label{e:cycle.specific.tree}
&(y_1,v_1) \to \cdots \to (y_s,v_s) \to (z_0,w_0), \\
&(y_1',v_1') \to \cdots \to (y_t',v_t') \to (z_0,w_0), \\
&(z_0,w_0) \to (z_1,w_1) \to \cdots \to (z_d,w_d) \to (0,u), \\
&(z_0,w_0) \to (z_1',w_1') \to \cdots \to (z_e',w_e') \to (0,u),
\end{split}
\end{align}
for some $s, t, d, e \ge 0$. Here, $(z_0,w_0)$ corresponds to a common node between $\ms{T}_1$ and $\ms{T}_2$. In the special case where $s=t=0$, the out-degree of $(z_0,w_0)$ is $0$. Moreover, if $d=0$ (resp.~$e=0$), the third (resp.~fourth) line in \eqref{e:cycle.specific.tree} simplifies to $(z_0,w_0) \to (0,u)$. In particular, if $d=e=0$, a  cycle structure does not occur. Indeed, \eqref{e:cycle.specific.tree} represents a cycle  only when either $d$ or $e$ is positive. The latter  is, in fact, a more challenging case, which we will discuss in detail below. For this discussion, we assume without loss of generality that $e\ge1$ (since $m\ge1$, it is possible to put this assumption). 
In the following, we shall consider three distinct situations based on the choices of $s$ and $t$, in addition to the aforementioned assumption $e\ge1$.
\vspace{3pt}

\noindent $(i)$ $s \wedge t \ge1$. \\
$(ii)$ $s\wedge t=0$ and $s\vee t\ge1$. \\
$(iii)$ $s =t=0$.
\vspace{3pt}

\noindent {\bfseries Case $(i)$:} $s \wedge t \ge1$. \\
Before identifying two  points in $\ms{T}_1$ and $\ms{T}_2$, the contribution to the variance from the  structure in \eqref{e:cycle.specific.tree} is asymptotically given by 
\begin{equation}  \label{e:initial.contribution.case(i)}
u^{-2\gamma} (\log u^{-1})^{s+t+d+e}. 
\end{equation}
After identifying two chosen vertices in $\ms{T}_1$ and $\ms{T}_2$, the contribution to the variance from the configuration in \eqref{e:cycle.specific.tree} can be updated as follows. 
\begin{align*}
A_u &:= \int_{w_0=u}^1 \int_{w_d=u}^{w_0} \int_{w_{d-1}=w_d}^{w_0} \cdots \int_{w_{1}=w_2}^{w_0} \int_{w_e'=u}^{w_0} \int_{w_{e-1}'=w_e'}^{w_0} \cdots \int_{w_1'=w_2'}^{w_0} w_0^{-2\gamma} (\log w_0^{-1})^{s+t-2} \\
&\quad \times \int_{z_0\in \R} \prod_{i=1}^d \int_{z_i\in \R} \prod_{j=1}^e \int_{z_j'\in \R} \one \big\{ (z_0,w_0) \to (z_1,w_1) \to \cdots \to (z_d,w_d) \to (0,u) \big\} \\
&\qquad \qquad\qquad\qquad\qquad\qquad\qquad\times \one \big\{ (z_0,w_0) \to (z_1',w_1') \to \cdots \to (z_e',w_e') \to (0,u) \big\}. 
\end{align*}
Here, $w_0^{-2\gamma} (\log w_0^{-1})^{s+t-2}$ represents the contribution from the structure in the first and second lines of \eqref{e:cycle.specific.tree} (see Lemma \ref{l:sequential.integral}). 
By ignoring the edge $(z_0,w_0)\to (z_1',w_1')$ and integrating indicators over all spatial coordinates, 
\begin{align*}
&\int_{z_0\in \R} \prod_{i=1}^d \int_{z_i\in \R} \prod_{j=1}^e \int_{z_j'\in \R} \one \big\{ (z_0,w_0) \to (z_1,w_1) \to \cdots \to (z_d,w_d) \to (0,u) \big\} \\
&\qquad \qquad\qquad\qquad\qquad\qquad\qquad\times \one \big\{ (z_0,w_0) \to (z_1',w_1') \to \cdots \to (z_e',w_e') \to (0,u) \big\} \\
&\le C u^{-2\gamma} w_0^{\gamma-1} \Big( \prod_{i=1}^d w_i^{-1} \Big) (w_1')^{\gamma-1} \Big\{ \prod_{j=2}^e (w_j')^{-1} \Big\}. 
\end{align*}
Because  of this bound, 
\begin{align*}
A_u &\le Cu^{-2\gamma} \int_{w_0=u}^1 w_0^{-1-\gamma} (\log w_0^{-1})^{s+t-2} \int_{w_d=u}^{w_0} w_d^{-1} \int_{w_{d-1}=w_d}^{w_0} w_{d-1}^{-1} \times \cdots \times  \int_{w_{1}=w_2}^{w_0} w_1^{-1} \\
&\quad \times  \int_{w_e'=u}^{w_0} (w_e')^{-1} \int_{w_{e-1}'=w_e'}^{w_0} (w_{e-1}')^{-1} \times \cdots \times \int_{w_2'=w_3'}^{w_0} (w_2')^{-1} \int_{w_1'=w_2'}^{w_0} (w_1')^{\gamma-1}
\end{align*}
(if $e=1$, the second line simplifies to $\int_{w_1'=u}^{w_0} (w_1')^{\gamma-1}$). 
Since $\int_{w_1'=w_2'}^{w_0} (w_1')^{\gamma-1} \le Cw_0^\gamma$, and 
\begin{align*}
&\int_{w_d=u}^{w_0} w_d^{-1} \int_{w_{d-1}=w_d}^{w_0} w_{d-1}^{-1} \times \cdots \times  \int_{w_{1}=w_2}^{w_0} w_1^{-1} \le C (\log u^{-1})^d, \\
&\int_{w_e'=u}^{w_0} (w_e')^{-1} \int_{w_{e-1}'=w_e'}^{w_0} (w_{e-1}')^{-1} \times \cdots \times \int_{w_2'=w_3'}^{w_0} (w_2')^{-1} \le C (\log u^{-1})^{e-1}, 
\end{align*}
we have 
\begin{equation}  \label{e:later.contribution.case(i)}
A_u \le Cu^{-2\gamma} (\log u^{-1})^{d+e-1} \int_{w_0=u}^1 w_0^{-1} (\log w_0^{-1})^{s+t-2} \le C u^{-2\gamma} (\log u^{-1})^{s+t+d+e-2}.  
\end{equation}
Thus, by \eqref{e:initial.contribution.case(i)} and \eqref{e:later.contribution.case(i)}, the corresponding terms in Var$\big( \inDt(u) \big)$ can be upper bounded by 
$$
Cu^{-2\ell\gamma} (\log u^{-1})^{2m} \frac{u^{-2\gamma} (\log u^{-1})^{s+t+d+e-2}}{u^{-2\gamma} (\log u^{-1})^{s+t+d+e}} = Cu^{-2\ell\gamma} (\log u^{-1})^{2m-2}\le Cu^{-2\ell\gamma} (\log u^{-1})^{2m-1}. 
$$

\noindent {\bfseries Case $(iii)$:} $s=t=0$. \\
Notice that the initial contribution to the variance from the structure in  \eqref{e:cycle.specific.tree} is asymptotically given by 
$$
u^{-2\gamma} (\log u^{-1})^{d+e}. 
$$
Instead of $A_u$ in Case $(i)$, we here need to estimate 
\begin{align*}
B_u &:= \int_{w_0=u}^1 \int_{w_d=u}^{w_0} \int_{w_{d-1}=w_d}^{w_0} \cdots \int_{w_{1}=w_2}^{w_0} \int_{w_e'=u}^{w_0} \int_{w_{e-1}'=w_e'}^{w_0} \cdots \int_{w_1'=w_2'}^{w_0} \\
&\quad \times \int_{z_0\in \R} \prod_{i=1}^d \int_{z_i\in \R} \prod_{j=1}^e \int_{z_j'\in \R} \one \big\{ (z_0,w_0) \to (z_1,w_1) \to \cdots \to (z_d,w_d) \to (0,u) \big\} \\
&\qquad \qquad\qquad\qquad\qquad\qquad\qquad\times \one \big\{ (z_0,w_0) \to (z_1',w_1') \to \cdots \to (z_e',w_e') \to (0,u) \big\} \\
&\le Cu^{-2\gamma} \int_{w_0=u}^1 w_0^{\gamma-1} \int_{w_d=u}^{w_0} w_d^{-1} \int_{w_{d-1}=w_d}^{w_0} w_{d-1}^{-1} \times \cdots \times  \int_{w_{1}=w_2}^{w_0} w_1^{-1} \\
&\quad \times  \int_{w_e'=u}^{w_0} (w_e')^{-1} \int_{w_{e-1}'=w_e'}^{w_0} (w_{e-1}')^{-1} \times \cdots \times \int_{w_2'=w_3'}^{w_0} (w_2')^{-1} \int_{w_1'=w_2'}^{w_0} (w_1')^{\gamma-1} \\
&\le Cu^{-2\gamma} (\log u^{-1})^{d+e-1}. 
\end{align*}
For the first inequality, we have again removed the condition $(z_0,w_0)\to (z_1',w_1')$ and computed integrals over all spatial coordinates. Then, the corresponding terms in  Var$\big( \inDt(u) \big)$ is at most 
$$
Cu^{-2\ell\gamma} (\log u^{-1})^{2m} \frac{u^{-2\gamma} (\log u^{-1})^{d+e-1}}{u^{-2\gamma} (\log u^{-1})^{d+e}} = Cu^{-2\ell\gamma} (\log u^{-1})^{2m-1}. 
$$

By repeating similar calculations, although the details are omitted here, we can conclude that in Case $(ii)$,  the corresponding terms in Var$\big( \inDt(u) \big)$  can be  bounded by  $Cu^{-2\ell\gamma} (\log u^{-1})^{2m-1}$. 
\medskip

It now remains to deal with the case  $m=0$. It then suffices to consider a configuration 
$$
(z_i,w_i) \to (0,u), \ \ i=1,\dots,2\ell-1. 
$$
In this case, our estimation for the variance is given by 
\begin{align*}
&\prod_{i=1}^{2\ell-1} \int_{w_i=u}^1 \prod_{j=1}^{2\ell-1} \int_{z_j\in \R} \one \big\{  (z_i,w_i)\to (0,u), \ i=1,\dots,2\ell-1\big\} \\
&= C u^{-(2\ell-1)\gamma} \Big( \int_u^1 w^{\gamma-1}\dif w \Big)^{2\ell-1} \le C u^{-(2\ell-1)\gamma}, 
\end{align*}
as desired. 

\subsection{Proof of Proposition \ref{p:exp.var.asym.clique}}

\noindent \underline{\textit{Proof of $(i)$}}: It follows from the Mecke formula that 
\begin{align}
\begin{split}  \label{e:nu.start}
\nu(u) &:= \int_{v=u}^1 \prod_{i=1}^{m-2} \int_{w_i=u}^v \int_{y\in \R} \prod_{j=1}^{m-2} \int_{z_j\in \R} \one \big\{ (y,v) \to (0,u) \big\} \\
&\qquad \qquad \times \one \big\{ (y,v)\to (z_\ell, w_\ell) \to (0,u), \, \ell=1,\dots,m-2 \big\} \\
&\qquad \qquad \times \one \big\{  \{ (z_1,w_1), \dots, (z_{m-2}, w_{m-2}) \} \text{ forms an } (m-2)\text{-clique} \big\}. 
\end{split}
\end{align}
We begin by approximating $\nu(u)$ as follows: 
\begin{align}
\begin{split}  \label{e:approx.nu}
\nu(u) &= \int_{v=u}^1 \prod_{i=1}^{m-2} \int_{w_i=u}^v \int_{y\in \R} \prod_{j=1}^{m-2} \int_{z_j\in \R} \one \big\{ (y,v) \to (0,u), \, (y,v)\to (z_\ell, w_\ell), \, \ell=1,\dots,m-2 \big\} \\
&\qquad \qquad \times \one \big\{  \{ (z_1,w_1), \dots, (z_{m-2}, w_{m-2}) \} \text{ forms an } (m-2)\text{-clique} \big\} + o(u^{-\gamma}), \ \ \ u\downarrow0, 
\end{split}
\end{align}
which indicates that  conditions $(z_\ell, w_\ell)\to (0,u)$, $\ell=1,\dots,m-2$, can be asymptotically removed from \eqref{e:nu.start}. For this purpose, we claim that for every $k=1,\dots,m-2$, 
\begin{align*}
A_u^{(k)} &:= \int_{v=u}^1 \prod_{i=1}^{m-2} \int_{w_i=u}^v \int_{y\in \R} \prod_{j=1}^{m-2} \int_{z_j\in \R} \one \big\{ (y,v) \to (0,u), \, (y,v)\to (z_\ell, w_\ell), \, \ell=1,\dots,m-2 \big\} \\ 
&\qquad \qquad \times \one \big\{ (z_k, w_k) \nrightarrow (0,u) \big\} = o(u^{-\gamma}), \ \ \text{as } u\downarrow0. 
\end{align*}
Below, we prove, without loss of generality, that $A_u^{(1)} = o(u^{-\gamma})$ as $u\downarrow0$. Notice that $(z_1, w_1) \nrightarrow (0,u)$ implies $|z_1|> \beta u^{-\gamma} w_1^{\gamma-1} \ge \beta u^{-\gamma} v^{\gamma-1}$. 
By the triangle inequality,  either $|y| \ge (1-\vep_u)\beta u^{-\gamma}v^{\gamma-1}$ or $|z_1-y|\ge \vep_u \beta u^{-\gamma}v^{\gamma-1}$ follows, where $\vep_u := (\log u^{-1})^{-1}$. 
\vspace{5pt}

\noindent {\bfseries Case $(i)$:} $|y|\ge (1-\vep_u) \beta u^{-\gamma}v^{\gamma-1}$. \\
Replacing the condition $(z_1, w_1) \nrightarrow (0,u)$ in $A_u^{(1)}$ with $|y|\ge (1-\vep_u) \beta u^{-\gamma}v^{\gamma-1}$, it is enough to demonstrate that 
\begin{align*}
B_u &:= \int_{w_{m-2}=u}^1 \dots \int_{w_1=u}^1 \int_{v=w_1\vee \cdots \vee w_{m-2}}^1 \int_{y\in \R} \prod_{j=1}^{m-2}\int_{z_j\in \R} \one \big\{    (1-\vep_u) \beta u^{-\gamma}v^{\gamma-1} \le |y| \le \beta u^{-\gamma}v^{\gamma-1} \big\} \\
&\qquad \qquad \qquad \times \one \big\{  |z_\ell-y| \le \beta w_\ell^{-\gamma} v^{\gamma-1} , \, \ell=1,\dots,m-2\big\} = o(u^{-\gamma}), \ \ \ u\downarrow0. 
\end{align*}
By integrating above  indicators with respect to spatial coordinates, 
\begin{align*}
B_u &= (2\beta)^{m-1}\vep_u u^{-\gamma} \int_{w_{m-2}=u}^1 w_{m-2}^{-\gamma}\times \dots\times  \int_{w_1=u}^1 w_1^{-\gamma} \int_{v=w_1\vee \cdots  \vee w_{m-2}}^1 v^{-(m-1)(1-\gamma)} \\
&= (2\beta)^{m-1} (m-2)! \vep_u u^{-\gamma} \int_{w_{m-2}=u}^1 w_{m-2}^{-\gamma} \int_{w_{m-3}=w_{m-2}}^1 w_{m-3}^{-\gamma} \times \\
&\qquad \qquad \qquad \qquad \qquad \cdots \times \int_{w_1=w_2}^1 w_1^{-\gamma}\int_{v=w_1}^1 v^{-(m-1)(1-\gamma)}. 
\end{align*}
Since $\vep_u u^{-\gamma} = o(u^{-\gamma})$ as $u \downarrow 0$, it suffices to verify that the multiple integral above is at most of constant order. For this purpose, we only need to consider the case where $\gamma$ is sufficiently close to $1/2$. Even in such cases, an elementary calculation shows that $B_u \le C \vep_u u^{-\gamma} = o(u^{-\gamma})$ as $u \downarrow 0$.
\vspace{3pt}

\noindent {\bfseries Case $(ii)$:} $|z_1-y|\ge \vep_u \beta u^{-\gamma}v^{\gamma-1}$. \\
Replacing the condition $(z_1, w_1) \nrightarrow (0,u)$ in $A_u^{(1)}$ with $|z_1-y|\ge \vep_u \beta u^{-\gamma}v^{\gamma-1}$, it is here sufficient to show that 
\begin{align*}
D_u &:= \int_{v=u}^1 \prod_{i=1}^{m-2} \int_{w_i=u}^v \int_{y\in \R} \prod_{j=1}^{m-2} \int_{z_j\in \R} \one \big\{ |y|\le \beta u^{-\gamma}v^{\gamma-1}, \, \vep_u \beta u^{-\gamma} v^{\gamma-1} \le |z_1-y| \le \beta w_1^{-\gamma} v^{\gamma-1}  \big\}  \\
&\qquad \qquad  \times \one \big\{ |z_\ell-y|\le \beta w_\ell^{-\gamma} v^{\gamma-1}, \, \ell=2,\dots,m-2 \big\}\times \one \{ \beta w_1^{-\gamma} v^{\gamma-1} \ge \vep_u \beta u^{-\gamma} v^{\gamma-1} \} = o(u^{-\gamma}). 
\end{align*}
To prove this claim, note first that 
$$
\one \{ \beta w_1^{-\gamma} v^{\gamma-1} \ge \vep_u \beta u^{-\gamma} v^{\gamma-1} \} = \one \{ w_1\le \vep_u^{-1/\gamma}u \}. 
$$
We shall consider the case that $w_1$ is the $j$th largest variable among $w_1,\dots,w_{m-2}$ for some $j\in \{1,\dots,m-2\}$. Assuming without loss of generality that $w_{m-2} \ge \cdots \ge w_{m-j}\ge w_1 \ge w_{m-j-1}\ge \cdots \ge  w_2$, it is enough to verify that 
\begin{align*}
E_u &:= \int_{v=u}^1 \int_{w_{m-2}=u}^v  \cdots  \int_{w_{m-j}=u}^v  \int_{w_{m-j-1}=u}^{\vep_u^{-1/\gamma}u}   \cdots  \int_{w_{1}=u}^{\vep_u^{-1/\gamma}u} \int_{y\in \R} \prod_{j=1}^{m-2}\int_{z_j\in \R} \\
&\qquad \one \big\{ |y|\le \beta u^{-\gamma}v^{\gamma-1}, \,|z_\ell-y|\le \beta w_\ell^{-\gamma} v^{\gamma-1}, \, \ell=1,\dots,m-2 \big\}  = o(u^{-\gamma}). 
\end{align*}
Integrating over all spatial coordinates, it turns out that 
\begin{align*}
E_u &= C u^{-\gamma} \int_{v=u}^1 v^{-(m-1)(1-\gamma)} \Big( \int_{w=u}^v w^{-\gamma} \Big)^{j-1} \Big( \int_{w=u}^{\vep_u^{-1/\gamma}u} w^{-\gamma} \Big)^{m-j-1} \\
&\le C \vep_u^{-\frac{1-\gamma}{\gamma} (m-j-1)} u^{-\gamma + (m-j-1)(1-\gamma)} \int_{v=u}^1 v^{-(m-j)(1-\gamma)}. 
\end{align*}
Evidently, we have 
$$
\int_{v=u}^1 v^{-(m-j)(1-\gamma)} \le \begin{cases}
C \log u^{-1} & \text{ if } \frac{m-j-1}{m-j} \le \gamma <1, \\
C u^{1-(m-j)(1-\gamma)} & \text{ if } \frac{1}{2} <\gamma <\frac{m-j-1}{m-j}. 
\end{cases}
$$
Therefore, if $\frac{m-j-1}{m-j} \le \gamma <1$, then 
$$
E_u \le C \vep_u^{-\frac{1-\gamma}{\gamma} (m-j-1)} u^{-\gamma + (m-j-1)(1-\gamma)} \log u^{-1} = o(u^{-\gamma}), \ \ \ u\downarrow0, 
$$
and if $ \frac{1}{2} <\gamma <\frac{m-j-1}{m-j}$, then 
$$
E_u \le C \vep_u^{-\frac{1-\gamma}{\gamma} (m-j-1)}= o(u^{-\gamma}), \ \ \ u\downarrow0. 
$$
Thus, \eqref{e:approx.nu} has been established. Next, it turns out that the leading term in \eqref{e:approx.nu} is equal to 
\begin{equation}  \label{e:after.approx}
\int_{v=u}^1 \int_{y\in \R} \one \big\{ |y|\le \beta u^{-\gamma}v^{\gamma-1} \big\} I(y,v), 
\end{equation}
where 
\begin{align*}
I(y,v) &:= \prod_{i=1}^{m-2} \int_{w_i=u}^v \prod_{j=1}^{m-2}\int_{z_j\in \R} \one \big\{ |z_\ell-y|\le \beta w_\ell^{-\gamma} v^{\gamma-1}, \, \ell=1,\dots,m-2 \big\} \\
&\qquad \qquad \qquad \times \one \big\{   |z_k-z_\ell| \le \beta (w_k\wedge w_\ell)^{-\gamma} (w_k \vee w_\ell)^{\gamma-1}, \, k, \ell \in \{ 1,\dots,m-2 \} \big\}. 
\end{align*}
Performing the change of variables by $z_{\ell}' = z_\ell-y$, $\ell=1,\dots,m-2$, which is  followed by additional change of variables $w_\ell=v\alpha_\ell$ and $z_\ell' = v^{-1}x_\ell$ for $\ell=1,\dots,m-2$, 
one can readily see that $I(y,v) = \frac{\gamma}{2\beta} C_{\beta,\gamma}$  regardless of the values of $y$ and $v$. Now, we can conclude that \eqref{e:after.approx}  equals $C_{\beta,\gamma}u^{-\gamma}$ as $u\downarrow0$. 
\medskip

\noindent \underline{\textit{Proof of $(ii)$}}: For ease of description, we denote $\inCc(u)$ as 
$$
\inCc(u) = \sum_{(P_1,\dots,P_{m-1})\in \mP_{\neq}^{m-1}} g\big( (0,u), P_1,\dots,P_{m-1} \big), 
$$
where 
\begin{align}  
\begin{split}  \label{e:def.g}
&g\big( (0,u), P_1,\dots,P_{m-1} \big) :=\one \big\{ (P_1,\dots,P_{m-1}) \text{ forms an } (m-1)\text{-clique}, \\ 
&\qquad \qquad\qquad\qquad \qquad\qquad\qquad \qquad P_i \to (0,u), \, i=1,\dots, m-1 \big\}. 
\end{split}
\end{align}
It follows from this expression, along with the Mecke formula for Poisson point processes, that 
\begin{align*}
\text{Var}\big( \inCc(u) \big) &= \nu(u) + \E \bigg[ \sum_{\substack{(P_1,\dots,P_{m-1})\in \mP_{\neq}^{m-1}}}\hspace{-10pt} \sum_{\substack{(Q_1,\dots,Q_{m-1})\in \mP_{\neq}^{m-1}, \\ |(P_1,\dots,P_{m-1})\cap (Q_1,\dots,Q_{m-1})|\le m-2}} \\
&\qquad \qquad \qquad \qquad \qquad  g\big( (0,u), P_1,\dots, P_{m-1} \big)\, g\big( (0,u), Q_1,\dots,Q_{m-1} \big)  \bigg] - \big( \nu(u) \big)^2  \\
&= \nu(u) + \sum_{q=1}^{m-2} \E \bigg[ \sum_{\substack{(P_1,\dots,P_{m-1})\in \mP_{\neq}^{m-1}}}\hspace{-10pt} \sum_{\substack{(Q_1,\dots,Q_{m-1})\in \mP_{\neq}^{m-1}, \\ |(P_1,\dots,P_{m-1})\cap (Q_1,\dots,Q_{m-1})|=q}} \\
&\qquad \qquad \qquad \qquad \qquad g\big( (0,u), P_1,\dots, P_{m-1} \big)\, g\big( (0,u), Q_1,\dots,Q_{m-1} \big)  \bigg] \\
&=: \nu(u) + \sum_{q=1}^{m-2} I_q (u). 
\end{align*}
Here, $q = \big|(P_1,\dots,P_{m-1}) \cap (Q_1,\dots,Q_{m-1})\big| \in \{ 1, \dots, m-2 \}$ represents the number of common nodes between $(P_1,\dots,P_{m-1})$ and $(Q_1,\dots,Q_{m-1})$. We define the following configuration of nodes. Let $(z,w)$ be the node of the highest mark among all $q$ common points, and let $(y_1, v_1), \dots, (y_{m-2}, v_{m-2})$ with $v_1 \ge \dots \ge v_{m-2}$ denote a set of points that, along with $(0,u)$ and $(z,w)$, form one of the $m$-cliques. Additionally, let $(z_1, w_1), \dots, (z_{m-q-1}, w_{m-q-1})$ with $w_1 \ge \dots \ge w_{m-q-1}$ represent the remaining nodes in the other $m$-clique, excluding $(0,u)$, $(z,w)$, and the other $q-1$ common points between the two cliques. Our discussion proceeds by dividing into two cases: $(i)$ $v_1 \wedge w_1 > w$ and $(ii)$ $v_1 \wedge w_1 \le w \le v_1 \vee w_1$ or $w \ge v_1 \vee w_1$.
\medskip

\noindent {\bfseries Case $(i)$:} $v_1 \wedge w_1 > w$. \\
As shown in Figure \ref{fig:clique.configuration.var}, we consider the following configuration of nodes and edge connections. 
\begin{align}
\begin{split}  \label{e:conf.v1.w1.larger.than.w}
&(y_1,v_1) \to \cdots \to (y_{m-2},v_{m-2}) \to (0,u), \\
&(z_1,w_1)\to \cdots \to (z_{m-q-1}, w_{m-q-1}), \\
&(y_1,v_1) \to (z,w), \ (z_1,w_1) \to (z,w). 
\end{split}
\end{align}
\begin{figure}
\includegraphics[scale=0.4]{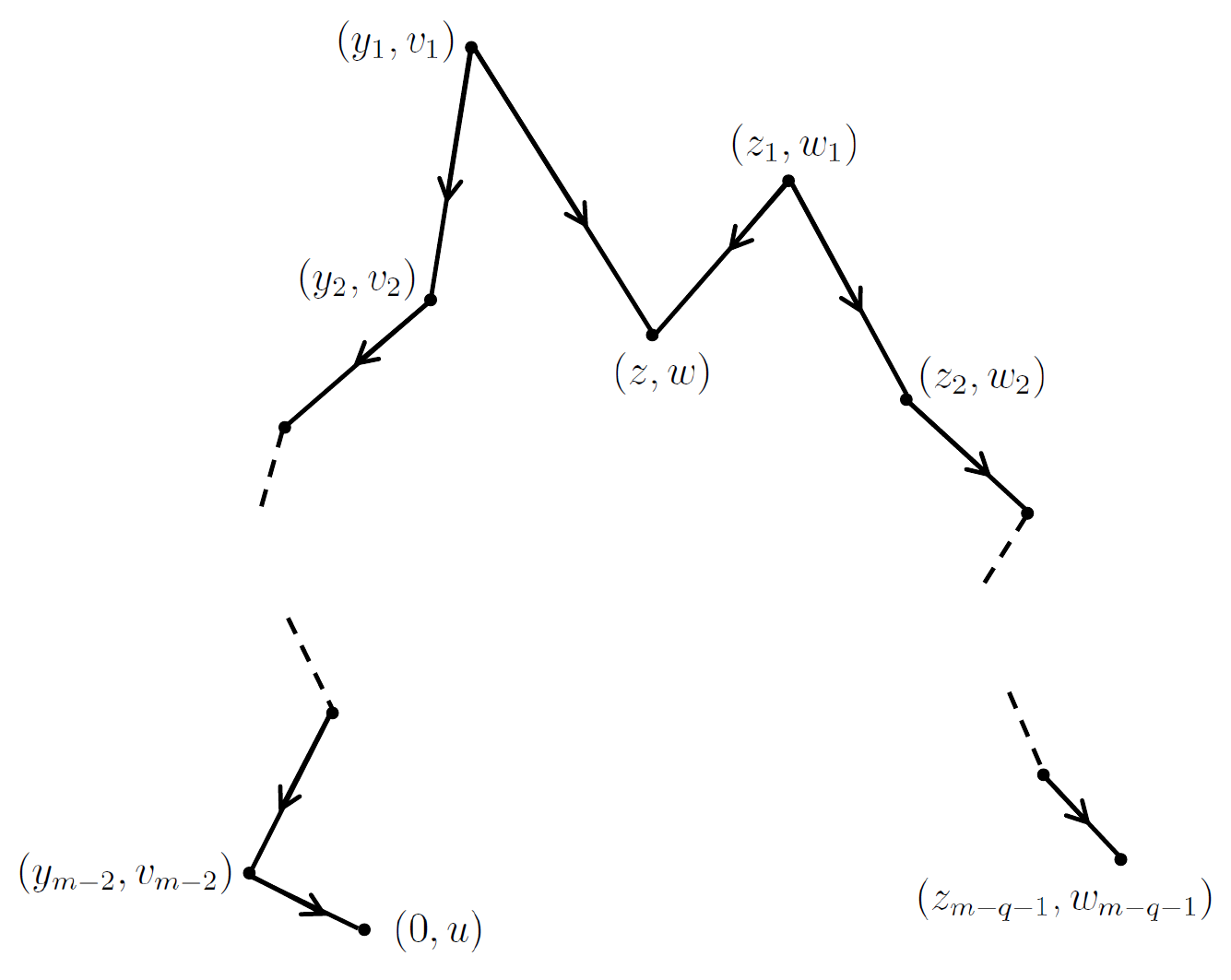}
\caption{\label{fig:clique.configuration.var} \footnotesize{Configuration of nodes and edges for the proof of Case $(i)$.}}
\end{figure}
We first deal with the case $m\ge4$. Then, the corresponding terms in $I_q(u)$ are upper bounded by 
\begin{align*}
A_u &:= \int_{v_{m-2}=u}^1 \int_{v_{m-3}=v_{m-2}}^1 \cdots \int_{v_2=v_3}^1  \int_{w_{m-q-1}=u}^1 \int_{w_{m-q-2}=w_{m-q-1}}^1 \cdots \int_{w_2=w_3}^1  \\
&\quad \times \int_{w=u}^1 \int_{v_1=w}^1 \int_{w_1=w}^1 \prod_{i=1}^{m-2} \int_{y_i \in \R} \int_{z\in \R} \prod_{j=1}^{m-2}\int_{z_j\in \R} \one \big\{ (y_1,v_1) \to \cdots \to (y_{m-2},v_{m-2}) \to (0,u) \big\}  \\
&\qquad \qquad \times \one \big\{ (z_1,w_1)\to \cdots \to (z_{m-q-1}, w_{m-q-1}),\, (y_1,v_1) \to (z,w), \, (z_1,w_1) \to (z,w) \big\}. 
\end{align*}
By integrating it over all spatial coordinates, 
\begin{align*}
A_u &= Cu^{-\gamma} \int_{v_{m-2}=u}^1 v_{m-2}^{-1} \int_{v_{m-3}=v_{m-2}}^1 v_{m-3}^{-1}\times \cdots \times  \int_{v_2=v_3}^1 v_2^{-1} \\
&\quad \times \int_{w_{m-q-1}=u}^1 w_{m-q-1}^{-\gamma}\int_{w_{m-q-2}=w_{m-q-1}}^1 w_{m-q-2}^{-1} \times \cdots\times  \int_{w_2=w_3}^1 w_2^{-1} \\
&\quad \times \int_{w=u}^1 w^{-2\gamma} \int_{v_1=w}^1 v_1^{2(\gamma-1)} \int_{w_1=w}^1 w_1^{2(\gamma-1)}  \\
&\le Cu^{1-3\gamma}  \int_{v_{m-2}=u}^1 v_{m-2}^{-1} \int_{v_{m-3}=v_{m-2}}^1 v_{m-3}^{-1}\times \cdots \times  \int_{v_2=v_3}^1 v_2^{-1} \\
&\quad \times \int_{w_{m-q-1}=u}^1 w_{m-q-1}^{-\gamma}\int_{w_{m-q-2}=w_{m-q-1}}^1 w_{m-q-2}^{-1} \times \cdots\times  \int_{w_2=w_3}^1 w_2^{-1}\\
&\le Cu^{1-3\gamma}  \int_{v_{m-2}=u}^1 v_{m-2}^{-1} \int_{v_{m-3}=v_{m-2}}^1 v_{m-3}^{-1}\times \cdots \times  \int_{v_2=v_3}^1 v_2^{-1}. 
\end{align*}
By virtue of  the identity $\int x^{-1}(\log x)^\ell \dif x = (\log x )^{\ell+1}/(\ell+1)$ for $\ell=0,1,\dots$, 
$$
\int_{v_{m-2}=u}^1 v_{m-2}^{-1} \int_{v_{m-3}=v_{m-2}}^1 v_{m-3}^{-1}\times \cdots \times  \int_{v_2=v_3}^1 v_2^{-1} = C (\log u^{-1})^{m-3}. 
$$
It thus follows that $A_u \le Cu^{1-3\gamma} (\log u^{-1})^{m-3}$. 

Subsequently, we consider the case $m=3$; then $q$ must be $1$ because of the restriction $v_1 \wedge w_1 >v$, and 
\begin{align*}
A_u &= \int_{w=u}^1 \int_{v_1=w}^1 \int_{w_1=w}^1 \int_{y_1\in \R}\int_{z\in \R} \int_{z_1\in \R} \one \big\{ (y_1,v_1)\to (0,u), \, (y_1,v_1)\to (z,w), \, (z_1,w_1)\to (z,w) \big\} \\
&= (2\beta)^3u^{-\gamma} \int_{w=u}^1 w^{-2\gamma} \int_{v_1=w}^1 v_1^{2(\gamma-1)} \int_{w_1=w}^1  w_1^{\gamma-1} \le Cu^{1-3\gamma}. 
\end{align*}
\vspace{5pt}

\noindent {\bfseries Case $(ii)$:} $v_1 \wedge w_1 \le w \le v_1 \vee w_1$ or $w\ge v_1 \vee w_1$. \\
Under a given configuration of nodes, one can select a point $P \in \mP$ whose mark is higher than $u$, such that there are still at least  $2m-q-3$ points below $P$. Consequently, the  corresponding terms in $I_q(u)$ are  upper bounded by 
\begin{equation}  \label{e:clique.counts.bdd}
 \E \Big[ \sum_{P\in \mP \cap N^\uparrow (0,u)} \mP \big( N^\downarrow (P) \big)^{2m-q-3}\Big], 
\end{equation}
where 
\begin{align*}
N^\uparrow (0,u) &:= \big\{ (z,w)\in \bbT: (z,w)\to (0,u) \big\}, \\
N^\downarrow(P) &:= \big\{ (z,w)\in \bbT: P\to (z,w) \big\}. 
\end{align*}

By the Mecke formula, 
\begin{align*}
&\E \Big[ \sum_{P\in \mP \cap N^\uparrow (0,u)} \mP \big( N^\downarrow (P) \big)^{2m-q-3}\Big] = \int_0^1 \int_{\R} \one \big\{ (y,v)\to (0,u) \big\} \E \Big[ \mP\big(  N^\downarrow (y,v)\big)^{2m-q-3} \Big] \dif y \dif  v. 
\end{align*}
Here, it is easy to see that $\mP \big( N^\downarrow (y,v) \big)$ has a Poisson distribution with mean $2\beta/(1-\gamma)$, which in turn implies that the law of $\mP \big( N^\downarrow (y,v) \big)$  does not depend on the choice of $(y,v)$. Therefore, 
$$
\E \Big[ \sum_{P\in \mP \cap N^\uparrow (0,u)} \mP \big( N^\downarrow (P) \big)^{2m-q-3}\Big] \le C \int_0^1 \int_{\R} \one \big\{ (y,v)\to (0,u) \big\} \dif y \dif v \le Cu^{-\gamma}. 
$$ 

Combining Cases $(i)$ and $(ii)$ together, for every $q=1,\dots,m-2$, 
$$
I_q(u) \le C \Big( u^{1-3\gamma}(\log u^{-1})^{m-3} + u^{-\gamma} \Big) \le Cu^{1-3\gamma} (\log u^{-1})^{m-3}, \ \ \ u\in (0,1), 
$$
and hence, 
$$
\text{Var}\big( \inCc(u) \big) \le \nu(u) + \sum_{q=1}^{m-2}I_q(u) \le Cu^{1-3\gamma} (\log u^{-1})^{m-3}, \ \ \ u\in (0,1). 
$$

\bibliography{ADRCM}

\end{document}